%% file: takase20151015.tex
\documentclass[a4paper,10pt]{article}
\usepackage{amsmath}  
\usepackage{amssymb}            
\usepackage{amsxtra}
\usepackage{amscd}
\usepackage[mathscr]{euscript}

\everymath{\displaystyle} 

\newtheorem{thm}{Theorem}[subsection]
\newtheorem{prop}[thm]{Proposition}

\newtheorem{ex}[thm]{Example}
\newtheorem{rem}[thm]{Remark}
\newtheorem{conj}[thm]{Conjecture}
\newtheorem{hypo}[thm]{Hypothesis}

\newcommand{\npmod}[1]{\!\!\pmod{#1}}
\newcommand{\nnpmod}[1]{\!\!\!\!\pmod{#1}}
\newcommand{\wsphat}{\sphat\;}

\newenvironment{proof}{\par\noindent{\bf[Proof]}}%
                      {$\blacksquare$\noindent\par\vspace{0.5\baselineskip}}
                      {$\blacksquare$\par\noindent}

\font\b=cmr10 scaled \magstep3

\def\bigzerou{\smash{\lower0.7ex\hbox{\b 0}}}
\def\bigastl{\smash{\lower0.7ex\hbox{\b *}}}
\def\bigastu{\smash{\lower2.7ex\hbox{\b *}}}

\makeatletter
\renewcommand{\subsection}{\@startsection%
  {subsection}%
  {2}%
  {0mm}%
  {\baselineskip}%
  {-0.2\parindent}%
  {\normalfont\normalsize\upshape\bfseries}}%
\def\@dotsep{1.5}
\def\@pnumwidth{1em}
\makeatother

\title{Regular characters of $GL_n(O)$ and\\
       Weil representations over finite fields} 
\author{Koichi Takase}
\begin{document}   
\maketitle

\input{remark.tex}
\input{constrep.tex}

\input{sympschur.tex}
\input{weilrep.tex}

\input{ref.tex}
Sendai 980-0845, Japan\\
Miyagi University of Education\\
Department of Mathematics\\
e-mail:k-taka2@ipc.miyakyo-u.ac.jp
\end{document}

%% file: remark.tex
\section{Introduction}
\label{sec:introduction}
Let $F$ be a non-Archimedean local field, 
$O$ its integer ring, $\frak{p}$ the maximal ideal of $O$ which is
generated by $\varpi$. The residue class field $O/\frak{p}$ is
denoted by $\Bbb F$ which is a finite field of $q$ elements. 
Let us denote by $O_l$ the residue class ring
$O/\frak{p}^l$ for a positive integer $l$ so that $O_1=\Bbb F$. 
Fix a continuous unitary character $\tau$ of the additive group $F$
such that 
$$
 \{x\in F\mid \text{\rm $\tau(xy)=1$ for all $y\in O$}\}=O.
$$
Define a unitary character $\widehat\tau$ of the additive group $\Bbb
F$ by 
$\widehat\tau\left(\overline x\right)
 =\tau\left(\varpi^{-1}x\right)$. 

The irreducible unitary representations of
$GL_n(O)$ may play an important role in the harmonic analysis of 
$GL_n(F)$, because the group $GL_n(O)$ is a maximal compact subgroup
of the locally compact group $GL_n(F)$. 
Since each irreducible unitary
representation of  $GL_n(O)$ factors through a finite group $G_r=GL_n(O_r)$
for some positive integer $r$, the classification of the
irreducible unitary representations of $GL_n(O)$ 
is reduced to determine
the irreducible complex representations of the finite group $G_r$. The
classification for the case $r=1$ is given by J.A.Green
\cite{Green1955}. Although the complete classification
of the irreducible representations of $G_r$, 
for the case $r>1$, is a {\it wild problem} 
\cite[p.4417]{Stasinski2009}, a remarkable
partial classification is given by G.Hill
\cite{Hill1993,Hill1994,Hill1995-1,Hill1995-2}. 
Let us recall his approach. 

 For any integer $0<i<r$, let us denote by $K_i$ the kernel of the
canonical surjective group homomorphism $G_r\to G_i$. Then the mapping 
$1_n+\varpi^iX\pmod{\frak{p}^r}\mapsto X\pmod{\frak p}$ gives a
surjective group homomorphism of $K_i$ onto the additive group
$M_n(\Bbb F)$. For any $\Bbb F$-vector subspace $U$ of $M_n(\Bbb
F)$, let us denote by $K_i(U)$ the inverse image of $U$ under the
group homomorphism, which is a normal subgroup of $K_i$.

Let $l$ be the smallest integer such that $r/2\leq l$ and put 
$l+l^{\prime}=r$. Then $K_l$ is
isomorphic to the additive group $M_n(O_{l^{\prime}})$ via the group
isomorphism 
$1_n+\varpi^lX\pmod{\frak{p}^r}\mapsto X\pmod{\frak{p}^{l^{\prime}}}$. For
any $\beta\in M_n(O)$, let us denote by $\psi_{\beta}$ the character
of $K_l$ defined by 
$$
 \psi_{\beta}(k)
 =\tau\left(\varpi^{-l^{\prime}}\text{\rm tr}(X\beta)\right)
$$
for $k=1_n+\varpi^lX\pmod{\frak{p}^r}\in K_l$. Then 
$\beta\pmod{\frak{p}^{l^{\prime}}}\mapsto\psi_{\beta}$ gives an isomorphism
of the additive group $M_n(O_{l^{\prime}})$ to the character group of $K_l$.
For each $\psi_{\beta}$, let us denote by 
$\text{\rm Irr}(G_r\mid\psi_{\beta})$ 
the equivalence classes of the irreducible complex representation
$\pi$ of $G_r$ such that 
$$
 \langle\psi_{\beta},\pi\rangle_{K_l}
 =\dim_{\Bbb C}\text{\rm Hom}_{K_l}(\psi_{\beta},\pi)>0.
$$
Note that $GL_n(O)$-conjugation of $\beta$ gives the same 
$\text{\rm Irr}(G_r\mid\psi_{\beta})$. 
Note also that, for any $\gamma\in GL_n(O)$, we have  
$\psi_{\gamma\beta\gamma^{-1}}(k)=\psi_{\beta}(\gamma^{-1}k\gamma)$. So
let us denote by $G_r(\psi_{\beta})$ the subgroup of 
$\gamma\pmod{\frak{p}^r}\in G_r$ such that 
$\gamma\beta\gamma^{-1}\equiv\beta\pmod{\frak{p}^{l^{\prime}}}$. Note that 
$K_l$ is contained in $G_r(\psi_{\beta})$. Then
Clifford theory says that $\text{\rm Irr}(G_r\mid\psi_{\beta})$ is determined by
the set $\text{\rm Irr}(G_r(\psi_{\beta})\mid\psi_{\beta})$ 
of the equivalence class of irreducible complex
representation $\sigma$ of $G_r(\psi_{\beta})$ such that 
$\langle\psi_{\beta},\sigma\rangle_{K_l}>0$. More precisely, 
$$
 \sigma\mapsto\text{\rm Ind}_{G_r(\psi_{\beta})}^{G_r}\sigma
$$
gives a bijection of $\text{\rm Irr}(G_r(\psi_{\beta})\mid\psi_{\beta})$ 
onto $\text{\rm Irr}(G_r\mid\psi_{\beta})$. 
So Hill's approach is to determine the set 
$\text{\rm Irr}(G_r(\psi_{\beta})\mid\psi_{\beta})$ 
according to the specific property of 
$\overline\beta=\beta\pmod{\frak p}\in M_n(\Bbb F)$. 

Suppose that $\overline\beta\in M_n(\Bbb F)$ is regular, that is, the
characteristic polynomial $\chi_{\overline\beta}(t)$ 
of $\overline\beta$ is equal to its minimal polynomial. Then we have 
(\cite[Cor. 3.9]{Hill1995-2})
$$
 G_r(\psi_{\beta})
 =\left(O_r[\beta_r]\right)^{\times}K_{l^{\prime}}
$$
where $\beta_r=\beta\pmod{\frak{p}^r}\in M_n(O_r)$ 
(the residue class of $X\in M_n(O)$ modulo a power $\frak{p}^i$ with
$i>1$ is denoted by $X_i\in M_n(O_i)$ whenever there is no
risk of confusion). For the sake of simplicity let us denote 
$\mathcal{C}=\left(O_r[\beta_r]\right)^{\times}$. 
If $r$ is even, then $l^{\prime}=l$ and all the elements of 
$\text{\rm Irr}(G_r(\psi_{\beta})\mid\psi_{\beta})$ 
are one-dimensional which are determined by the extensions to 
$\mathcal{C}$ of $\psi_{\beta}|_{\mathcal{C}\cap K_l}$. 

Theorem 4.6 in his paper \cite{Hill1995-2}, Hill treats the case 
where $r=2l-1$ is odd and the regular
$\overline\beta\in M_n(\Bbb F)$ is split, that is the eigenvalues of 
$\overline\beta$ are contained in $\Bbb F$. 
Unfortunately the proof of the theorem works only for a semi-simple
$\overline\beta$ as shown by Proposition
\ref{prop:k-l-1-conjugate-of-psi-in-x-0-psi-beta} of this paper. 

Shintani \cite{Shintani1968} and G\'erardin \cite{Gerardin1972} treat
the case where $r=2l-1$ is odd and 
the characteristic polynomial of $\overline\beta$ is
irreducible over $\Bbb F$ (the cuspidal case). However their method is
rather complicated because it 
depends on the parity of $n$ and treats two cases separately. 

In this paper we will assume that $\overline\beta\in M_n(\Bbb F)$ is
regular and $r=2l-1$ is odd and will give a method to describe the set 
$\text{\rm Irr}(G_r(\psi_{\beta})\mid\psi_{\beta})$ 
in a uniform manner under a hypothesis that a certain
Schur multiplier associated with a symplectic space over finite field 
is trivial 
(Theorem \ref{th:description-of-regular-character-under-hypothesis}).
This hypothesis is valid if the characteristic polynomial 
of $\overline\beta$ is separable (Theorem
\ref{th:original-hypothesis-is-valid-if-chara-poly-is-separable}) so
that the cuspidal case is included in this case.
In the case of split $\overline\beta$, we will give evidences good 
enough to call the hypothesis a conjecture (Theorem 
\ref{th:original-hypothesis-is-valid-in-split-small-n}). 

Our basic idea is that the group $K_{l-1}$ is very close to the
Heisenberg group associated with a symplectic space over a finite
field. So we can construct an irreducible representation of $K_{l-1}$
via the Schr\"odinger representation. Then the theory of Weil
representation \cite{Weil1964} enable us to extend the irreducible
representation of $K_{l-1}$ to a representation of
$\mathcal{C}K_{l-1}$, but there appears an obstruction
to the extension (section 
\ref{sec:regular-character-via-weil-rep-over-finite-field}).
This obstruction is described by certain Schur
multiplier with which we will mainly concern in this paper. 
The section
\ref{sec:schur-multiplier-of-symplectic-space-over-finite-field} is 
devoted to the study of it. 

The characteristic of the finite field $\Bbb F$ is arbitrary in 
sections 
\ref{sec:remark-on-a-result-of-g-hill} and 
\ref{sec:general-description-of-the-regular-character}. It is assumed
to be odd in sections 
\ref{sec:schur-multiplier-of-symplectic-space-over-finite-field} and 
\ref{sec:regular-character-via-weil-rep-over-finite-field}. 

Acknowledgment: the author express his thanks to the referee for
reading carefully the manuscript and for giving many
suggestions to improve it. 

\section{Remarks on a result of Hill}
\label{sec:remark-on-a-result-of-g-hill}

\subsection{}
\label{subsec:remark-on-result-of-hill}
In this section, we will assume that $r=2l-1$ is odd and that 
$\overline\beta\in M_n(\Bbb F)$ is regular and split. 
So, after taking suitable
$GL_n(O)$-conjugate of $\beta$, we can assume that $\overline\beta$ is
in a Jordan canonical form
\begin{equation}
 \overline\beta
 =\begin{bmatrix}
   J_{n_1}(a_1)&      &            \\
               &\ddots&            \\
               &      &J_{n_f}(a_f)
  \end{bmatrix},
\label{eq:beta-mod-p-is-in-jordan-canonical-form}
\end{equation}
where $a_1,\cdots,a_f\in\Bbb F$ are the distinct eigenvalues of
$\overline\beta$ and 
\begin{equation*}
 J_m(a)=\begin{bmatrix}
         a&1& &      & \\
          &a&1&      & \\
          & &a&\ddots& \\
          & & &\ddots&1\\
          & & &      &a
        \end{bmatrix}\in M_m(\Bbb F)
%\label{eq:jordan-block}
\end{equation*}
is a Jordan block. 
Let $W$ be the $\Bbb F$-vector subspace of the upper triangular
matrices in $M_n(\Bbb F)$. Let us denote by $X(\psi_{\beta})$ the set
of the one-dimensional representation $\psi$ of $K_{l-1}(W)$ such that 
$\psi|_{K_l}=\psi_{\beta}$. Then we have 
$$
 \sharp X(\psi_{\beta})=(K_{l-1}(W):K_l)=|W|=q^{n(n+1)/2},
$$
because we have $K_l\subset K_{l-1}(W)$ and 
$[K_{l-1}(W),K_{l-1}(W)]\subset\text{\rm Ker}(\psi_{\beta})$. 
An action of $k\in K_{l-1}$ on $\psi\in X(\psi_{\beta})$ is defined 
by $\psi^k(x)=\psi(k xk^{-1})$. 
Then \cite{Hill1995-2} shows, in the proof of Theorem 4.6, that the
isotropy subgroup of any $\psi\in X(\psi_{\beta})$ is 
$K_{l-1}(W)$ and that 
$$
 \sharp\left(X(\psi_{\beta})/K_{l-1}\right)
 =\sharp\left(X(\psi_{\beta})\right)/(K_{l-1}:K_{l-1}(W))
 =q^{n(n+1)/2-n(n-1)/2}=q^n.
$$
On the other hand $\varepsilon\in\mathcal{C}$ also acts on 
$\psi\in X(\psi_{\beta})$ by 
$\psi^{\varepsilon}(x)=\psi(\varepsilon x\varepsilon^{-1})$, and 
let us denote by $X_0(\psi_{\beta})$ the set of 
$\psi\in X(\psi_{\beta})$ such that $\psi^{\varepsilon}=\psi$ for all
$\varepsilon\in\mathcal{C}$. Then it is also shown in the proof of
Theorem 4.6 of \cite{Hill1995-2} that the set $X_0(\psi_{\beta})$ has
$q^n$ elements (see Remark
\ref{remark:argument-of-hill-works-also-in-equal-char-case} below).

The
subgroup of $GL_n$ consisting of the diagonal elements is denoted by
$D_n$. For any $T\in D_n(O)$ we define a one-dimensional
representation $\omega_T$ of $K_{l-1}(W)$ by 
$$
 \omega_T(x)=\tau\left(\varpi^{-1}\text{\rm tr}(XT)\right)
 =\tau\left(\varpi^{-1}\sum_{i=1}^nX_{ii}T_{ii}\right)
$$
for $x=1_n+\varpi^{l-1}X\pmod{\frak{p}^r}\in K_{l-1}(W)$. Here we
denote by $A_{ij}$ the $(i,j)$-element of a matrix $A$ in general. 
Then 
$\omega_T^{\varepsilon}=\omega_T$ for all $\varepsilon\in\mathcal{C}$ 
because $\mathcal{C}\pmod{\frak{p}}\subset W$. Furthermore 
$\omega_T=\omega_{T^{\prime}}$ for $T,T^{\prime}\in D_n(O)$ if and
only if $T\equiv T^{\prime}\pmod{\frak{p}}$. In particular the number
of distinct $\omega_T$ is $q^n$. So for any 
$\psi, \psi^{\prime}\in X_0(\psi_{\beta})$, there exists unique 
$T\pmod{\frak p}\in D_n(\Bbb F)$ such that 
$\psi^{\prime}\psi^{-1}=\omega_T$. Here we put
$$
 T=\begin{bmatrix}
    T_1&      &   \\
       &\ddots&   \\
       &      &T_f
   \end{bmatrix}
$$
with $T_i\in D_{n_i}(O)$. Then we have the following proposition;

\begin{prop}\rm\label{prop:k-l-1-conjugate-of-psi-in-x-0-psi-beta}
$\psi$ and $\psi^{\prime}$ belong to the same $K_{l-1}$-orbit in
$X(\psi_{\beta})$ if
and only if $\text{\rm tr}\,T_i\equiv 0\pmod{\frak p}$ for all
$i=1,\cdots,f$. 
\end{prop}
\begin{proof}
Take any $k=1_n+\varpi^{l-1}Y\pmod{\frak{p}^r}\in K_{l-1}$. Then
$$
 k^{-1}=1_n-\varpi^{l-1}Y+\varpi^{2(l-1)}X^2\nnpmod{\frak{p}^r},
$$
so that 
$$
 kxk^{-1}x^{-1}
 =1_n+\varpi^{2(l-1)}(YX-XY)\nnpmod{\frak{p}^r}\in K_l
$$
for all $x=1_n+\varpi^{l-1}X\pmod{\frak{p}^r}\in K_{l-1}(W)$. Then 
$\psi^{\prime}=\psi^k$ if and only if 
\begin{align*}
 \omega_T(x)
 &=\psi(kxk^{-1}x^{-1})=\psi_{\beta}(kxk^{-1}x^{-1})\\
 &=\tau\left(\varpi^{-1}\text{\rm tr}((YX-XY)\beta)\right)
  =\tau\left(\varpi^{-1}\text{\rm tr}((\beta Y-Y\beta)X)\right)
\end{align*}
for all $x=1_n+\varpi^{l-1}X\pmod{\frak{p}^r}\in K_{l-1}(W)$, that is
$$
 \text{\rm tr}(XT)\equiv
 \text{\rm tr}(X(\beta Y-Y\beta))\pmod{\frak p}
$$
for all $X\in M_n(O)$ such that $X\pmod{\frak p}\in W$. This condition
is equivalent to
\begin{equation}
 \beta Y-Y\beta\equiv
 \begin{bmatrix}
     T_1   &      &        \\
           &\ddots&        \\
           &      & T_f
 \end{bmatrix}\pmod{\frak p}.
\label{eq:equivalent-condition-on-y}
\end{equation}
Decompose $Y$ into the blocks
$$
 T=\begin{bmatrix}
    Y_{11}&\cdots&Y_{1f}\\
    \vdots&\ddots&\vdots\\
    Y_{f1}&\cdots&Y_{ff}
   \end{bmatrix}
$$
where $Y_{ij}$ is a $n_i\times n_j$-matrix. Then the condition 
\eqref{eq:equivalent-condition-on-y} is equivalent to 
\begin{equation}
 J_{n_i}(a_i)Y_{ij}-Y_{ij}J_{n_j}(a_j)
 \equiv\begin{cases}
        T_i&:i=j,\\
         0 &:i>j
       \end{cases}
 \pmod{\frak p}
\label{eq:block-condition-on-y}
\end{equation}
for all $i,j=1,\cdots,f$ such that $i\geq j$. This means that 
$\text{\rm tr}(T_i)\equiv 0\pmod{\frak p}$ for all $i=1,\cdots,f$. 
After direct calculations, the conditions 
\eqref{eq:block-condition-on-y} are equivalent to 
$Y_{ij}=0$ if $i>j$ and
$$
 Y_{ii}\equiv
        \begin{bmatrix}
         \ast&       &      &              &        \\
          t_1& \ast  &      &                  &    \\
             &t_1+t_2&\ddots&                  &    \\
             &       &\ddots&    \ast          &    \\
             &  &      &t_1+\cdots+t_{n_i-1}&\ast
   \end{bmatrix}\pmod{\frak p}
$$
where $t_1,\cdots,t_{n_i}$ are the diagonal elements of $T_i$, so that 
$t_1+\cdots+t_{n_i}\equiv 0\pmod{\frak p}$.
\end{proof}

This proposition shows that the number of $K_{l-1}$-orbits in
$X(\psi_{\beta})$ which contains some elements of $X_0(\psi_{\beta})$
is $q^f$, where $f$ is the number of Jordan block in
$\overline\beta$. This means that the argument in the proof of Theorem
4.6 of \cite{Hill1995-2} works only in the case $f=n$, that is, the
case of semi-simple $\overline\beta$. 

\begin{rem}\label{remark:argument-of-hill-works-also-in-equal-char-case}
Although Hill works over $p$-adic fields in \cite{Hill1995-2}, the
arguments in the proof of Theorem 4.6 to show the equations 
$\sharp\left(X(\psi_{\beta})/K_{l-1}\right)=q^n$ and 
$\sharp X_0(\psi_{\beta})=q^n$ work well also in the case of equal
 characteristic local fields.
\end{rem}

\begin{rem}\label{remark:where-and-why-hill-fail}
As the referee points out, Hill's proof fails at p. 627 where he
claims that $\lambda\psi_a^{\prime}$ is a $C_G(\widehat a)$-stable
extension of 
$\psi_a$  to $H_a$. In fact this is not always the case, because 
the subgroup $N_2/K_l$ of $H_a/K_l$ is
not necessarily normalized by $C_G(\widehat a)$.
\end{rem}

\subsection{}
\label{subsec:suggestive-proposition}
We present here a proposition which is suggestive for the following
arguments.

\begin{prop}\rm\label{prop:siggestive-proposition}
Two elements $\psi$ and $\psi^{\prime}$ of $X(\psi_{\beta})$ belong to
the same $K_{l-1}$-orbit if and only if they have the same restriction
on $K_{l-1}(\Bbb F[\overline\beta])$.
\end{prop}
\begin{proof}
Assume that $\psi^{\prime}=\psi^k$ with a 
$k=1_n+\varpi^{l-1}Y\pmod{\frak{p}^r}\in K_{l-1}$. Then we have
\begin{align*}
 \psi^{\prime}(x)\psi(x)^{-1}
 &=\psi(kxk^{-1}x^{-1})=\psi_{\beta}(kxk^{-1}x^{-1})\\
 &=\tau\left(\varpi^{-1}\text{\rm tr}((YX-XY)\beta)\right)=1
\end{align*}
for all 
$x=1_n+\varpi^{l-1}\pmod{\frak{p}^r}
 \in K_{l-1}(\Bbb F[\overline\beta])$. The number of
elements in the $K_{l-1}$-orbit of $\psi$ is 
$$
 (K_{l-1}:K_{l-1}(W))=(M_n(\Bbb F):W)=q^{n(n-1)/2}.
$$
On the other hand, the number of elements of $X(\psi_{\beta})$ which
have the same restriction as $\psi$ on 
$K_{l-1}(\Bbb F[\overline\beta])$ is
$$
 (K_{l-1}(W):K_{l-1}(\Bbb F[\overline\beta]))
 =(W:\Bbb F[\overline\beta])=q^{n(n-1)/2}.
$$
\end{proof}

%% file: constrep.tex
\section{General description of the regular characters}
\label{sec:general-description-of-the-regular-character}
In this section, we will describe how Clifford theory works for our
problem concerned with $\beta\in M_n(O)$ such that 
$\overline\beta\in M_n(\Bbb F)$ is regular. 

\subsection{}
\label{sec:prop-on-rep-of-finite-group}
The arguments of the following sections are based upon an elementary
proposition on an irreducible representation of a finite group. For
the sake of completeness, we will include here the proposition. Its
proof is well-known and can be omitted. 
See \cite[Prop.8.3.3]{BushnellFrohlich}. 

Let $G$ be a finite group and $N$ a normal subgroup of $G$ such that
$G/N$ is commutative. Let $\psi$ be a group homomorphism of $N$ to 
$\Bbb C^{\times}$ such that $\psi(gng^{-1})=\psi(n)$ for all 
$g\in G$ and $n\in N$. Then the mapping $D_{\psi}$ of 
$G/N\times G/N$ to $\Bbb C^{\times}$ is well-defined by 
$$
 D_{\psi}(\overline x,\overline y)=\psi([x,y]),
$$
where $\overline x=x\pmod{N}\in G/N$ and $[x,y]=xyx^{-1}y^{-1}$. In
fact we have
\begin{align*}
 \psi([xn,ym])
 &=\psi(xnymn^{-1}x^{-1}m^{-1}y^{-1})\\
 &=\psi(xy(y^{-1}nymn^{-1}x^{-1}m^{-1}x)x^{-1}y^{-1})\\
 &=\psi(xy(y^{-1}nymn^{-1}x^{-1}m^{-1}x)(xy)^{-1}[x,y])\\
 &=\psi(y^{-1}ny)\psi(mn^{-1})\psi(x^{-1}m^{-1}x)\psi([x,y])
  =\psi([x,y])
\end{align*}
for all $x,y\in G$ and $m,n\in N$. Then we have 
$D_{\psi}(\overline y,\overline x)
 =D_{\psi}(\overline x,\overline y)^{-1}$ and 
$\overline x\mapsto D_{\psi}(\overline x,\ast)$ is a group
homomorphism of $G/N$ to the character group $(G/N)\sphat$ of the
commutative group $G/N$. Then we have

\begin{prop}\rm\label{prop:irred-rep-of-finite-group}
Suppose that $D_{\psi}$ is non-degenerate, that is, 
$\overline x\mapsto D_{\psi}(\overline x,\ast)$ is a group isomorphism
of $G/N$ onto $(G/N)\sphat$. Then $G$ has unique irreducible
representation $\pi_{\psi}$ up to equivalence 
such that $\langle\psi,\pi_{\psi}\rangle_N>0$. In this case 
$$
 \text{\rm Ind}_N^G\psi=\bigoplus^{\dim\pi_{\psi}}\pi_{\psi}
$$
and $\pi_{\psi}(n)$ is the homothety $\psi(n)$ for all $n\in N$.
\end{prop}

\subsection{}
\label{sec:schur-multiplier}
Fix a $\beta\in M_n(O)$ such that $\overline\beta\in M_n(\Bbb F)$ is
regular. The purpose of this subsection is to describe the set 
$\text{\rm Irr}(G_r\mid\psi_{\beta})$ (Theorem
\ref{th:description-of-regular-character-under-hypothesis}) by means
of Clifford theory. Since 
$$
 [K_{l-1}(\Bbb F[\overline\beta]),K_{l-1}(\Bbb F[\overline\beta])]
 \subset[K_{l-1}(W),K_{l-1}(W)]
 \subset\text{\rm Ker}(\psi_{\beta})
$$
the set $Y(\psi_{\beta})$ of the group homomorphisms $\psi$ of 
$K_{l-1}(\Bbb F[\overline\beta])$ to 
$\Bbb C^{\times}$ such that $\psi|_{K_l}=\psi_{\beta}$ is not
empty. Take a $\psi\in Y(\psi_{\beta})$. 

Define an alternating bilinear form $\langle\,,\rangle_{\overline\beta}$
on $M_n(\Bbb F)$ by 
$$
 \langle X\nnpmod{\frak p},Y\nnpmod{\frak p}\rangle_{\overline\beta}
 =\text{\rm tr}\left((XY-YX)\beta\right)
 \nnpmod{\frak p}.
$$ 
By virtue of the regularity of $\overline\beta\in M_n(\Bbb F)$, we have
$$
 \{X\in M_n(\Bbb F)\mid\langle X,Y\rangle_{\overline\beta}=0\,
    \text{\rm for all $Y\in M_n(\Bbb F)$}\}
 =\Bbb F[\overline\beta].
$$
In other word, the $\Bbb F$-vector space 
$\Bbb V_{\beta}=M_n(\Bbb F)/\Bbb F[\overline\beta]$ is a symplectic
$\Bbb F$-space with respect to a symplectic form 
$(\dot X,\dot Y)\mapsto\langle X,Y\rangle_{\overline\beta}$ where 
$\dot X=X\pmod{\Bbb F[\overline\beta]}\in\Bbb V_{\beta}$ with 
$X\in M_n(\Bbb F)$.

We will apply Proposition \ref{prop:irred-rep-of-finite-group} to the
groups $G=K_{l-1}$ and $N=K_{l-1}(\Bbb F[\overline\beta])$. 
First of all, the mapping 
$1_n+\varpi^{l-1}X\npmod{\frak{p}^r}\mapsto X\npmod{\frak p}$ 
induces an isomorphism of $G/N$ onto $\Bbb V_{\beta}$. 
In particular $G/N$ is commutative.  Then, for any
$x=1_n+\varpi^{l-1}X\npmod{\frak{p}^r}
 \in K_{l-1}(\Bbb F[\overline\beta])$ and 
$k=1_n+\varpi^{l-1}Y\npmod{\frak{p}^r}\in K_{l-1}$, we have
$$
 \psi(kxk^{-1})\psi(x)^{-1}=\psi_{\beta}(kxk^{-1}x^{-1})
 =\widehat\tau\left(\text{\rm tr}((YX-XY)\beta\right)
 =1,
$$
that is $\psi(kxk^{-1})=\psi(x)$. Now define 
$D_{\psi}(\overline x,\overline y)=\psi([x,y])$ for 
$\overline x,\overline y\in G/N$. If we write 
$$
 x=1_n+\varpi^{l-1}X\nnpmod{\frak{p}^r},\;\;
 y=1_n+\varpi^{l-1}Y\nnpmod{\frak{p}^r}\in K_{l-1},
$$
then we have
$$
 D_{\psi}(\overline x,\overline y)
 =\psi_{\beta}(xyx^{-1}y^{-1})
 =\tau\left(\varpi^{-1}
   \langle\overline X,\overline Y\rangle_{\overline\beta}\right),
$$
so that $D_{\psi}$ is non-degenerate. Hence, by Proposition 
\ref{prop:irred-rep-of-finite-group}, there exists a unique, up to
equivalence, irreducible representation $\pi_{\psi}$ of $K_{l-1}$ with
representation space $H_{\psi}$ such that 
$\langle\psi,\pi_{\psi}\rangle_{K_{l-1}(\Bbb F[\overline\beta])}>0$. 
Furthermore, for any $x\in K_{l-1}(\Bbb F[\overline\beta])$, the
operator $\pi_{\psi}(x)$ is the homothety of $\psi(x)$. 

Take any $\varepsilon\in\mathcal{C}$. Then for any
$$
 x=1_n+\varpi^{l-1}X\nnpmod{\frak{p}^r}
  \in K_{l-1}(\Bbb F[\overline\beta])
$$
we have
$$
 \varepsilon x\varepsilon^{-1}x^{-1}
 \equiv
 1_n+\varpi^{l-1}(\varepsilon X\varepsilon^{-1}-X)
    +\varpi^{2(l-1)}(X^2-\varepsilon X\varepsilon^{-1}X)
 \pmod{\frak{p}^r}.
$$
Here $x\in K_{l-1}(\Bbb F[\overline\beta])$ means that 
$X\npmod{\frak p}\in\Bbb F[\overline\beta]$ so that 
$\varepsilon X\varepsilon^{-1}\equiv X\npmod{\frak p}$. Then
$\varepsilon x\varepsilon^{-1}x^{-1}\in K_l$ and we have
\begin{align*}
 &\psi(\varepsilon x\varepsilon^{-1})\psi(x)^{-1}
  =\psi_{\beta}(\varepsilon x\varepsilon^{-1}x^{-1})\\
 &=\tau\left(\varpi^{-l}\text{\rm tr}(
              (\varepsilon X\varepsilon^{-1}-X)\beta)
    +\varpi^{-1}\text{\rm tr}
              ((X^2-\varepsilon X\varepsilon^{-1}X)\beta)\right)
  =1,
\end{align*}
that is $\psi(\varepsilon x\varepsilon^{-1})=\psi(x)$ for all 
$x\in K_{l-1}(\Bbb F[\overline\beta])$. Then 
$\pi_{\psi}^{\varepsilon}(k)=\pi_{\psi}(\varepsilon k\varepsilon^{-1})$ defines an
irreducible representation of $K_{l-1}$ on $H_{\psi}$ 
equivalent to $\pi_{\psi}$. So there
exists a $U(\varepsilon)\in GL_{\Bbb C}(H_{\psi})$ such that 
\begin{equation}
 \pi_{\psi}(\varepsilon k\varepsilon^{-1})
 =U(\varepsilon)\circ\pi_{\psi}(k)\circ U(\varepsilon)^{-1}
\label{eq:conjugate-formula-of-pi-psi}
\end{equation}
for all $k\in K_{l-1}$.
Let us denote by $c_U$ the two cocycle associated with $U$, that is 
$c_U(\varepsilon,\eta)\in\Bbb C^{\times}$ such that
\begin{equation}
 U(\varepsilon)\circ U(\eta)
 =c_U(\varepsilon,\eta)\cdot U(\varepsilon\eta)
\label{eq:2-cocycle-of-u}
\end{equation}
for all $\varepsilon, \eta\in\mathcal{C}$. Then the cohomology class 
$c(\psi)\in H^2(\mathcal{C},\Bbb C^{\times})$ of $c_U$ is 
independent of the choice of $U(\varepsilon)$ for each 
$\varepsilon\in\mathcal{C}$. Here $\mathcal{C}$ acts trivially 
on $\Bbb C^{\times}$. 

Now we will propose the following hypothesis;

\begin{hypo}\label{hyp:schur-multiplier-c-psi-is-trivial-for-all-psi}
The Schur multiplier $c(\psi)\in H^2(\mathcal{C},\Bbb C^{\times})$ is
trivial for all $\psi\in Y(\psi_{\beta})$.
\end{hypo}

In section \ref{sec:regular-character-via-weil-rep-over-finite-field}, 
when $\text{\rm ch}\,\Bbb F\neq 2$, we will describe the 2-cocycle
$c_U$ much more explicitly (the formula 
\eqref{eq:explicit-formula-of-schur-multiplier-associated-with-rep}) 
by means of the Schr\"odinger
representation and the Weil representation associated with the
symplectic space $\Bbb V_{\beta}$ over finite field $\Bbb F$, and will 
show that 
\begin{enumerate}
\item Hypothesis
      \ref{hyp:schur-multiplier-c-psi-is-trivial-for-all-psi} is valid
      if the characteristic polynomial of $\overline\beta\in M_n(\Bbb F)$  is
      separable (Theorem
      \ref{th:original-hypothesis-is-valid-if-chara-poly-is-separable}), 
\item we have enough evidences to believe that Hypothesis 
      \ref{hyp:schur-multiplier-c-psi-is-trivial-for-all-psi} is valid
      if $\overline\beta$ is split and $\text{\rm ch}\,\Bbb F$ is big
      enough (Theorem
      \ref{th:original-hypothesis-is-valid-in-split-small-n} and the
      remark after the theorem).
\end{enumerate}

Now we will assume Hypothesis
\ref{hyp:schur-multiplier-c-psi-is-trivial-for-all-psi}. 
Then there exists an 
extension $\widetilde\pi_{\psi}$ of $\pi_{\psi}$ to
$\mathcal{C}K_{l-1}$. In fact, we can assume that $U$ is a group
homomorphism of $\mathcal{C}$ to $GL_{\Bbb C}(H_{\psi})$. For any 
$\varepsilon\in\mathcal{C}\cap K_{l-1}$, there exists a 
$\theta(\varepsilon)\in\Bbb C^{\times}$ such that 
$\pi_{\psi}(\varepsilon)
 =\theta(\varepsilon)\cdot U(\varepsilon)$. 
Then $\theta$ is a character of $\mathcal{C}\cap K_{l-1}$ which has an
extension $\Theta$ to the commutative group $\mathcal{C}$. Now 
$\widetilde\pi_{\psi}(\varepsilon k)
 =\Theta(\varepsilon)\cdot U(\varepsilon)\circ\pi_{\psi}(k)$ 
($\varepsilon\in\mathcal{C}, k\in K_{l-1}$) is a well-defined
 extension of $\pi_{\psi}$ to $\mathcal{C}K_{l-1}$. 

Our purpose is to determine the set 
$\text{\rm Irr}(G_r(\psi_{\beta})\mid\psi_{\beta})$. Take any 
$\sigma\in\text{\rm Irr}(G_r(\psi_{\beta})\mid\psi_{\beta})$. Then 
$$
 \sigma\hookrightarrow
 \text{\rm Ind}_{K_l}^{\mathcal{C}K_{l-1}}\psi_{\beta}
 =\text{\rm Ind}_{K_{l-1}}^{\mathcal{C}K_{l-1}}\left(
   \text{\rm Ind}_{K_l}^{K_{l-1}}\psi_{\beta}\right).
$$
We have 
$$
 \text{\rm Ind}_{K_l}^{K_{l-1}}\psi_{\beta}
 =\text{\rm Ind}_{K_{l-1}(\Bbb F[\overline\beta])}^{K_{l-1}}\left(
   \text{\rm Ind}_{K_l}^{K_{l-1}(\Bbb F[\overline\beta])}
    \psi_{\beta}\right)
$$
and 
$$
 \text{\rm Ind}_{K_l}^{K_{l-1}(\Bbb F[\overline\beta])}\psi_{\beta}
 =\bigoplus_{\psi\in Y(\psi_{\beta})}\psi
$$
because 
$[K_{l-1}(\Bbb F[\overline\beta]),K_{l-1}(\Bbb F[\overline\beta])]
 \subset\text{\rm Ker}(\psi_{\beta})$. For each 
$\psi\in Y(\psi_{\beta})$, we have
$$
 \text{\rm Ind}_{K_{l-1}(\Bbb F[\overline\beta])}^{K_{l-1}}\psi
 =\bigoplus^{\dim\pi_{\psi}}\pi_{\psi}
$$
by Proposition \ref{prop:irred-rep-of-finite-group}. Hence there
exists a $\psi\in Y(\psi_{\beta})$ such that 
$\sigma\hookrightarrow
 \text{\rm Ind}_{K_{l-1}}^{\mathcal{C}K_{l-1}}\pi_{\psi}$. Now we have 
$$
 \text{\rm Ind}_{K_{l-1}}^{\mathcal{C}K_{l-1}}\pi_{\psi}
 =\bigoplus_{\chi\in\text{\rm Irr}(\mathcal{C}/\mathcal{C}\cap K_{l-1})}
   \chi\otimes\widetilde\pi_{\psi}
$$
where $\text{\rm Irr}\left(\mathcal{C}/\mathcal{C}\cap K_{l-1}\right)$
is the character group of the abelian group 
$\mathcal{C}/\mathcal{C}\cap K_{l-1}$ which is identified with the
 one-dimensional representations of $\mathcal{C}K_{l-1}$ trivial on 
$K_{l-1}$. Hence $\{\chi\otimes\widetilde\pi_{\psi}\}_{\chi}$ are the
 extensions of $\pi_{\psi}$ to $\mathcal{C}K_{l-1}$ and 
$\sigma=\chi\otimes\widetilde\pi_{\psi}$ for some 
$\chi\in\text{\rm Irr}\left(\mathcal{C}/\mathcal{C}\cap K_{l-1}\right)$. 
Then the Clifford theory gives 

\begin{thm}\label{th:description-of-regular-character-under-hypothesis}
Under the Hypothesis 
\ref{hyp:schur-multiplier-c-psi-is-trivial-for-all-psi}, a bijection 
$$
 \text{\rm Irr}\left(\mathcal{C}/\mathcal{C}\cap K_{l-1}\right)
                                  \times Y(\psi_{\beta})\to
 \text{\rm Irr}(G_r\mid\psi_{\beta})
$$
is given by 
$(\chi,\psi)\mapsto
 \text{\rm Ind}_{\mathcal{C}K_{l-1}}^{G_r}(\chi\otimes\widetilde\pi_{\psi})$.
\end{thm}

%% file: sympschur.tex
\section{Schur multiplier associated with symplectic space over finite
  field}
\label{sec:schur-multiplier-of-symplectic-space-over-finite-field}
In this section the finite field $\Bbb F$ is supposed to be of odd
characteristic. Our purpose is to define and to study 
a Schur multiplier associated with a symplectic space over $\Bbb
F$. The results of this section are used in the next section to
describe the Schur multiplier in Hypothesis 
\ref{hyp:schur-multiplier-c-psi-is-trivial-for-all-psi}.  
Fix a regular $\overline\beta\in M_n(\Bbb F)$. 

\subsection{}
\label{subsec:general-setting-of-finite-symplectic-schur-multiplier}
Due to the regularity of $\overline\beta\in M_n(\Bbb F)$,  we have
$$
 \{X\in M_n(\Bbb F)\mid X\overline\beta=\overline\beta X\}=\Bbb F[\overline\beta].
$$
Then $\Bbb V_{\beta}=M_n(\Bbb F)/\Bbb F[\overline\beta]$ is a symplectic
space over $\Bbb F$ with respect to the symplectic form 
$$
 \langle\dot X,\dot Y\rangle_{\overline\beta}=\text{\rm tr}((XY-YX)\overline\beta).
$$
Here we denote $\dot X=X\pmod{\Bbb F[\overline\beta]}\in\Bbb V_{\beta}$ for 
$X\in M_n(\Bbb F)$. 
Take a character $\rho:\Bbb F[\overline\beta]\to\Bbb C^{\times}$ of the
additive group $\Bbb F[\overline\beta]$. Let 
$v\mapsto[v]$ be a $\Bbb F$-linear splitting of the exact sequence
\begin{equation}
 0\to\Bbb F[\overline\beta]\to M_n(\Bbb F)\to\Bbb V_{\beta}\to 0
\label{eq:canonical-exact-seq-of-v-beta}
\end{equation}
of $\Bbb F$-vector space. In other word let us choose a 
$\Bbb F$-vector subspace $V\subset M_n(\Bbb F)$ such that 
$M_n(\Bbb F)=V\oplus\Bbb F[\overline\beta]$, and put 
$v=[v]\pmod{\Bbb F[\overline\beta]}$ with $[v]\in V$ 
for $v\in\Bbb V_{\beta}$. For any 
$\varepsilon\in\Bbb F[\overline\beta]^{\times}$ and 
$v=\dot X\in\Bbb V_{\beta}$, put 
$\varepsilon^{-1}v\varepsilon
 =\varepsilon^{-1}X\varepsilon\pmod{\Bbb F[\overline\beta]}\in\Bbb V_{\beta}$
 and 
$$
 \gamma(v,\varepsilon)
 =\varepsilon^{-1}[v]\varepsilon-[\varepsilon^{-1}v\varepsilon]
 \in\Bbb F[\overline\beta].
$$
Since $v\mapsto\rho(\gamma(v,\varepsilon))$ is an additive character of
$\Bbb V_{\beta}$, there exists uniquely a 
$v_{\varepsilon}\in\Bbb V_{\beta}$ such that 
$$
 \rho(\gamma(v,\varepsilon))
 =\widehat\tau\left(\langle v,v_{\varepsilon}\rangle_{\overline\beta}\right)
$$
for all $v\in\Bbb V_{\beta}$. Then, for any 
$\varepsilon, \eta\in\Bbb F[\overline\beta]^{\times}$, we have
\begin{equation}
 v_{\varepsilon\eta}
 =\varepsilon v_{\eta}\varepsilon^{-1}+v_{\varepsilon}
\label{eq:multiplication-formula-of-v-epsilon}
\end{equation}
because 
$$
 \gamma(v,\varepsilon\eta)
 =\gamma(v,\varepsilon)+\gamma(\varepsilon^{-1}v\varepsilon,\eta)
$$
for all $v\in\Bbb V_{\beta}$. Put
$$
 c_{\overline\beta,\rho}(\varepsilon,\eta)
 =\widehat\tau\left(2^{-1}\langle v_{\varepsilon},
                    v_{\varepsilon\eta}\rangle_{\overline\beta}\right)
$$
for $\varepsilon, \eta\in\Bbb F[\overline\beta]^{\times}$. Then the relation 
\eqref{eq:multiplication-formula-of-v-epsilon} shows that 
$c_{\overline\beta,\rho}\in Z^2(\Bbb F[\overline\beta]^{\times},\Bbb C^{\times})$ is a
      2-cocycle with trivial action of $\Bbb F[\overline\beta]^{\times}$ on 
      $\Bbb C^{\times}$. Furthermore we have

\begin{prop}
\label{prop:schur-multiplier-associated-with-finite-symplectic-space}
 The cohomology class 
      $[c_{\overline\beta,\rho}]\in H^2(\Bbb F[\overline\beta]^{\times},\Bbb C^{\times})$ is
      independent of the choice of the $\Bbb F$-linear splitting 
      $v\mapsto[v]$.
\end{prop}
\begin{proof}
Take another $\Bbb F$-linear splitting $v\mapsto[v]^{\prime}$ 
with respect to which
we will define $\gamma^{\prime}(v,\varepsilon)\in\Bbb F[\overline\beta]$ 
and $v^{\prime}_{\varepsilon}\in\Bbb V_{\beta}$ as above. 
Then there exists a $\delta\in\Bbb V_{\beta}$ such that 
$\rho([v]-[v]^{\prime})
 =\widehat\tau\left(\langle v,\delta\rangle_{\overline\beta}\right)$ 
for all $v\in\Bbb V_{\beta}$. We have 
$v_{\varepsilon}^{\prime}
 =v_{\varepsilon}+\delta-\varepsilon^{-1}\delta\varepsilon$ for all 
$\varepsilon\in\Bbb F[\overline\beta]^{\times}$. So if we put
$$
 \alpha(\varepsilon)
 =\widehat\tau\left(2^{-1}\langle 
    v_{\varepsilon}^{\prime}-v_{\varepsilon^{-1}},\delta\rangle_{\overline\beta}
            \right)
$$
for $\varepsilon\in\Bbb F[\overline\beta]^{\times}$, then we have
$$
 \widehat\tau\left(2^{-1}\langle v_{\varepsilon}^{\prime},
           v_{\varepsilon\eta}^{\prime}\rangle_{\overline\beta}\right)
 =\widehat\tau\left(2^{-1}\langle v_{\varepsilon},
           v_{\varepsilon\eta}\rangle_{\overline\beta}\right)\cdot
  \alpha(\eta)\alpha(\varepsilon\eta)^{-1}\alpha(\varepsilon)
$$
for all $\varepsilon, \eta\in\Bbb F[\overline\beta]^{\times}$.
\end{proof}

Let $H(\Bbb V_{\beta})=\Bbb V_{\beta}\times\Bbb C^1$ be the Heisenberg
group associated with the symplectic $\Bbb F$-space $\Bbb
V_{\beta}$. Here $\Bbb C^1$ is the multiplicative group of the complex
number with absolute value one. 
The group operation on $H(\Bbb V_{\beta})$ is defined by 
$$
 (u,s)\cdot(v,t)
 =(u+v,s\cdot t\cdot
    \widehat\tau\left(2^{-1}\langle u,v\rangle_{\overline\beta}\right)).
$$
There exists uniquely up to isomorphism an irreducible representation
$(\pi_{\overline\beta},H_{\overline\beta})$ of $H(\Bbb V_{\beta})$ such that 
$\pi_{\overline\beta}(0,s)=s$ for all $(0,s)\in H(\Bbb V_{\beta})$. We
call the representation $\pi_{\overline\beta}$ the Schr\"odinger representation
of the Heisenberg group $H(\Bbb V_{\beta})$. For
any $\sigma\in Sp(\Bbb V_{\beta})$, the mapping 
$(u,s)\mapsto(u\sigma,s)$ is an automorphism of $H(\Bbb V_{\beta})$ so
that there exists a $T(\sigma)\in GL_{\Bbb C}(H_{\overline\beta})$ such that 
$$
 \pi_{\overline\beta}(u\sigma,s)
 =T(\sigma)^{-1}\circ\pi_{\overline\beta}(u,s)\circ T(\sigma)
$$
for all $(u,s)\in H(\Bbb V_{\beta})$. For any 
$\sigma, \sigma^{\prime}\in Sp(\Bbb V_{\beta})$, there exists a 
$c_T(\sigma,\sigma^{\prime})\in\Bbb C^{\times}$ such that
$$
 T(\sigma)\circ T(\sigma^{\prime})
 =c_T(\sigma,\sigma^{\prime})\cdot T(\sigma\sigma^{\prime}).
$$
Then $c_T\in Z^2(Sp(\Bbb V_{\beta}),\Bbb C^{\times})$ is a 
2-cocycle with trivial action of $Sp(\Bbb V_{\beta})$ on 
$\Bbb C^{\times}$, and the cohomology class 
$[c_T]\in H^2(Sp(\Bbb V_{\beta}),\Bbb C^{\times})$ is independent of
the choice of each $T(\sigma)$. We have a group homomorphism 
$\varepsilon\mapsto\sigma_{\varepsilon}$ of $\Bbb F[\overline\beta]^{\times}$
to $Sp(\Bbb V_{\beta})$ defined by 
$\sigma_{\varepsilon}(v)=\varepsilon^{-1}v\varepsilon$. 
So a 2-cocycle $c_T\in Z^2(\Bbb F[\overline\beta]^{\times},\Bbb C^{\times})$
and a cohomology class $[c_T]\in H^2(\Bbb F[\overline\beta]^{\times},\Bbb C^{\times})$
is defined by 
$c_T(\varepsilon,\eta)=c_T(\sigma_{\varepsilon},\sigma_{\eta})$. 
Put 
$$
 U(\varepsilon)
 =\pi_{\overline\beta}(v_{\varepsilon},0)\circ T(\sigma_{\varepsilon})
 \in GL_{\Bbb C}(H_{\overline\beta})
$$
for $\varepsilon\in\Bbb F[\overline\beta]^{\times}$. Then 
\eqref{eq:multiplication-formula-of-v-epsilon} gives
$$
 U(\varepsilon)\circ U(\eta)
 =c_{\overline\beta,\rho}(\varepsilon,\eta)\cdot c_T(\varepsilon,\eta)\cdot
  U(\varepsilon\eta)
$$
for all $\varepsilon, \eta\in\Bbb F[\overline\beta]^{\times}$. 
Now we will investigate the validity of the following hypothesis

\begin{hypo}\label{hypo:tirviality-of-schur-multiplier}
The cohomology class 
$[c_{\overline\beta,\rho}c_T]\in H^2(\Bbb F[\overline\beta]^{\times},\Bbb C^{\times})$ is trivial 
for all additive characters $\rho$ of $\Bbb F[\overline\beta]$.
\end{hypo}

\begin{rem}\label{remark:assertion-hypothesis-really-assert}
If $\rho$ is trivial, then $c_{\overline\beta,\rho}(\varepsilon,\eta)=1$ for
all $\varepsilon, \eta\in\Bbb F[\overline\beta]^{\times}$. So Hypothesis 
\ref{hypo:tirviality-of-schur-multiplier} asserts that the cohomology
class $[c_T]\in H^2(\Bbb F[\overline\beta]^{\times},\Bbb C^{\times})$ is
trivial and that the cohomology classes 
$[c_{\overline\beta,\rho}]\in H^2(\Bbb F[\overline\beta]^{\times},\Bbb C^{\times})$ are
trivial for all $\rho$ .
\end{rem}

\subsection{}
\label{subsec:schrodinger-rep-associated-with-polarization}
The Schr\"odinger representation of the Heisenberg group 
$H(\Bbb V_{\beta})$ is realized as follows. Let 
$\Bbb V_{\beta}=\Bbb W^{\prime}\oplus\Bbb W$ be a polarization of the
symplectic $\Bbb F$-space $\Bbb V_{\beta}$. Let $L^2(\Bbb W^{\prime})$
be the complex vector space of the complex-valued functions $f$ on 
$\Bbb W^{\prime}$ with norm 
$|f|^2=\sum_{w\in\Bbb W^{\prime}}|f(w)|^2$. The Schr\"odinger
representation $\pi_{\overline\beta}$ of $H(\Bbb V_{\beta})$ is realized 
on $L^2(\Bbb W^{\prime})$ by
$$
 \left(\pi_{\overline\beta}(u,s)f\right)(w)
 =s\cdot\widehat\tau\left(2^{-1}\langle u_-,u_+\rangle_{\overline\beta}
                        +\langle w,u_+\rangle_{\overline\beta}\right)\cdot
  f(w+u_-)
$$
for $(u,s)\in H(\Bbb V_{\beta})$ where 
$u=u_-+u_+$ with $u_-\in\Bbb W^{\prime}, u_+\in\Bbb W$. 
We call it the Schr\"odinger representation associated with the
polarization $\Bbb V_{\beta}=\Bbb W^{\prime}\oplus\Bbb W$.

The method of \cite{Lion-Vergne} and \cite{Rao} enable us to give an
explicit description of 2-cocycle 
$c_T\in Z^2(\Bbb F[\overline\beta]^{\times},\Bbb C^{\times})$. Each element 
$\sigma\in Sp(\Bbb V_{\beta})$ is denoted by blocks 
$\begin{pmatrix}
  a&b\\
  c&d
 \end{pmatrix}$ with
respect to the polarization $\Bbb V_{\beta}=\Bbb W^{\prime}\oplus\Bbb W$ 
(that is $a\in\text{\rm End}_{\Bbb F}(\Bbb W^{\prime}), 
b\in\text{\rm Hom}_{\Bbb F}(\Bbb W^{\prime},\Bbb W)$ etc). Then a unitary
operator $T_{\widehat\tau}(\sigma)$ of $L^2(\Bbb W^{\prime})$ is defined by
\begin{align*}
 &\left(T_{\widehat\tau}(\sigma)f\right)(w)\\
 =&\text{\rm const.}\times\sum_{v\in\Bbb W/\text{\rm Ker c}}
    f(wa+vc)\cdot\widehat\tau\left(
      2^{-1}\langle wa+vc,wb+vd\rangle_{\overline\beta}
      -2^{-1}\langle w,v\rangle_{\overline\beta}\right)
\end{align*}
with a positive real number $\text{\rm consts.}$  and  we have
$$
 \pi_{\widehat\tau}(u\sigma,s)
 =T_{\widehat\tau}(\sigma)^{-1}\circ\pi_{\widehat\tau}(u,s)\circ 
  T_{\widehat\tau}(\sigma).
$$
Defined in \cite{Lion-Vergne} and \cite{Rao} is the generalized Weil
constant $\gamma_{\widehat\tau}(Q)\in\Bbb C^{\times}$ associated with a 
$\Bbb F$-quadratic form $Q$ having the following properties;
\begin{enumerate}
\item if $Q$ is an orthogonal sum of two $\Bbb F$-quadratic forms
         $Q_1$ and $Q_2$, then 
         $\gamma_{\widehat\tau}(Q)=\gamma_{\widehat\tau}(Q_1)\cdot\gamma(Q_2)$,
\item $\gamma_{\widehat\tau}(0)=1$, and 
      if $Q$ is a regular $\Bbb F$-quadratic form on a $\Bbb
      F$-vector space $X$
$$
 \gamma_{\widehat\tau}(Q)
 =q^{-2^{-1}\dim_{\Bbb F}X}\sum_{x\in X}\widehat\tau(Q(x)).
$$
\end{enumerate}
Particularly for $a\in\Bbb F^{\times}$, we have
$$
 \gamma_{\widehat\tau}(a)
 =q^{-1/2}\sum_{x\in\Bbb F}\widehat\tau(ax^2)
 =\left(\frac a{\Bbb F}\right)\cdot q^{-1/2}
  \sum_{x\in\Bbb F^{\times}}\left(\frac x{\Bbb F}\right)\cdot\widehat\tau(x)
$$
where 
$$
 \left(\frac a{\Bbb F}\right)
 =\begin{cases}
   1&:\text{\rm if $a$ is a square in $\Bbb F$},\\
  -1&:\text{\rm if $a$ is not a square in $\Bbb F$}
  \end{cases}
$$
is the Legendre symbol with respect to the field $\Bbb F$. 
Now for every 
$\sigma, \sigma^{\prime}\in Sp(\Bbb V_{\beta})$, let us denote by 
$Q_{\sigma,\sigma^{\prime}}$ the $\Bbb F$-quadratic form on 
$\Bbb W^{\prime}\times\Bbb W^{\prime}\times\Bbb W^{\prime}$ defined by
$$
 (u,v,w)\mapsto\langle u,v\sigma^{\prime}\rangle_{\overline\beta}
              +\langle v,w\sigma\rangle_{\overline\beta}
              -\langle u,w\sigma\sigma^{\prime}\rangle_{\overline\beta}.
$$
Put 
$c_{\widehat\tau}(\sigma,\sigma^{\prime})
 =\gamma_{\widehat\tau}(Q_{\sigma,\sigma^{\prime}})^{-1}$, then 
\cite{Lion-Vergne} and \cite{Rao} give the formula
$$
 T_{\widehat\tau}(\sigma)\circ T_{\widehat\tau}(\sigma^{\prime})
 =c_{\widehat\tau}(\sigma,\sigma^{\prime})\cdot 
  T_{\widehat\tau}(\sigma\sigma^{\prime}).
$$
So if we put $T(\varepsilon)=T_{\widehat\tau}(\sigma_{\varepsilon})$ for
$\varepsilon\in\Bbb F[\overline\beta]^{\times}$, then we have 
$$
 c_T(\varepsilon,\eta)=c_{\widehat\tau}(\sigma_{\varepsilon},\sigma_{\eta})
 =\gamma_{\widehat\tau}(Q_{\varepsilon,\eta})^{-1}
$$
for $\varepsilon, \eta\in\Bbb F[\overline\beta]^{\times}$. Here we 
put $Q_{\varepsilon,\eta}=Q_{\sigma_{\varepsilon},\sigma_{\eta}}$. Put
$$
 B(\Bbb V_{\beta})
 =\left\{\begin{pmatrix}
          a&b\\
          0&d
         \end{pmatrix}\in Sp(\Bbb V_{\beta})\right\}
$$
which is a subgroup of $Sp(\Bbb V_{\beta})$. For any 
$\sigma=\begin{pmatrix}
         a&b\\
         0&d
        \end{pmatrix}\in B(\Bbb V_{\beta})$, we have
$$
 \left(T_{\widehat\tau}(\sigma)f\right)(w)
 =\widehat\tau\left(2^{-1}\langle wa,wb\rangle_{\overline\beta}\right)
   \cdot f(wa)
$$
for $f\in L^2(\Bbb W^{\prime})$, and 
$\sigma\mapsto T_{\widehat\tau}(\sigma)$ is a group homomorphism of
$B(\Bbb V_{\beta})$ in $GL_{\Bbb C}(L^2(\Bbb W^{\prime}))$. Hence 
$c_{\widehat\tau}(\sigma,\sigma^{\prime})=1$ for all 
$\sigma, \sigma^{\prime}\in B(\Bbb V_{\beta})$. In other word, if 
$\Bbb W\sigma_{\varepsilon}=\Bbb W, \Bbb W\sigma_{\eta}=\Bbb W$ for 
$\varepsilon, \eta\in\Bbb F[\overline\beta]^{\times}$, then 
$c_T(\varepsilon,\eta)=1$. 

\subsection{}
\label{subsec:symplectic-schur-multiplier-reduced-to-basic-cases}
Let $\chi_{\overline\beta}(t)=\prod_{i=1}^rp_i(t)^{e_i}$ 
be the irreducible decomposition of the characteristic polynomial 
$\chi_{\overline\beta}(t)$ of $\overline\beta$ with distinct
irreducible polynomials $p_i(t)$ in $\Bbb F[t]$ ($e_i>0$). So we can assume
that $\overline\beta$ is of the form
$$
 \overline\beta=\begin{bmatrix}
         \overline\beta_1&      &         \\
                &\ddots&         \\
                &      &\overline\beta_r
        \end{bmatrix}
$$
where $\overline\beta_i\in M_{n_i}(t)$ is a regular element with characteristic
polynomial $p_i(t)^{e_i}$, so that $n_i=e_i\deg p_i(t)$. Then we have
$$
 \Bbb F[\overline\beta]
 =\left\{\begin{bmatrix}
          X_1&      &         \\
             &\ddots&         \\
             &      &X_r
         \end{bmatrix}\biggm| X_i\in\Bbb F[\overline\beta_i]\right\}.
$$
Define $\Bbb F$-vector subspaces $L_{\beta}, M_{\beta}$ of $M_n(\Bbb F)$ by
$$
 L_{\beta}
 =\left\{\begin{bmatrix}
          X_1&      &         \\
             &\ddots&         \\
             &      &X_r
         \end{bmatrix}\biggm| X_i\in M_{n_i}(\Bbb F)\right\}
$$
and
$$
 M_{\beta}
 =\left\{\begin{bmatrix}
          O_{n_1}&\cdots&\ast   \\
          \vdots &\ddots&\vdots \\
          \ast   &\cdots&O_{n_r}
         \end{bmatrix}\in M_n(\Bbb F)\right\}
$$
respectively.  Then we have 
$$
 \Bbb V_{\beta}=L_{\beta}/\Bbb F[\overline\beta]\oplus M_{\beta}
 =\bigoplus_{i=1}^r\Bbb V_{\beta_i}\oplus M_{\beta}.
$$
Here $M_{\beta}$ is a symplectic $\Bbb F$-space with respect to 
the symplectic form 
$\langle X,Y\rangle_{\overline\beta}=\text{\rm tr}((XY-YX)\overline\beta)$ with a
polarization $M_{\beta}=M_{\beta}^-\oplus M_{\beta}^+$ defined by 
$$
 M_{\beta}^+
 =\left\{\begin{bmatrix}
          O_{n_1}&\cdots&\ast     \\
          \vdots &\ddots&\vdots   \\
             0   &\cdots&O_{n_r}
         \end{bmatrix}\in M_n(\Bbb F)\right\},
$$
and
$$
 M_{\beta}^-
 =\left\{\begin{bmatrix}
          O_{n_1}&\cdots&    0    \\
          \vdots &\ddots&\vdots  \\
          \ast   &\cdots&O_{n_r}
         \end{bmatrix}\in M_n(\Bbb F)\right\}.
$$
Let $(\pi_{\overline\beta,M},L^2(M_{\beta}^-))$ be the Schr\"odinger
representation of the Heisenberg group 
$H(M_{\beta})=M_{\beta}\times\Bbb C^1$ associated with the polarization 
$M_{\beta}=M_{\beta}^-\oplus M_{\beta}^+$. Any 
$\varepsilon\in\Bbb F[\overline\beta]^{\times}$ defines an element 
$X\mapsto\varepsilon^{-1}X\varepsilon$ of $Sp(M_{\beta})$ which keeps
the subspace $M_{\beta}^{\pm}$ stable. So define a 
$T_M(\varepsilon)\in GL_{\Bbb C}(L^2(M_{\beta}^-))$ by
$$
 \left(T_M(\varepsilon)f\right)(W)
 =f\left(\varepsilon^{-1}W\varepsilon\right).
$$
Then $\varepsilon\mapsto T_M(\varepsilon)$ is a group homomorphism of
$\Bbb F[\overline\beta]^{\times}$ such that
$$
 \pi_{\overline\beta,M}(\varepsilon^{-1}X\varepsilon,s)
 =T_M(\varepsilon)^{-1}\circ\pi_{\overline\beta,M}(X,s)\circ
  T_M(\varepsilon)
$$
for all $\varepsilon\in\Bbb F[\overline\beta]^{\times}$ and 
$(X,s)\in H(M_{\beta})$. 

The symplectic form
$\langle\,,\rangle_{\overline\beta}$ on $\Bbb V_{\beta}$ induces the symplectic
form $\langle\,,\rangle_{\overline\beta_i}$ on $\Bbb V_{\beta_i}$. Let us
denote $(\pi_{\overline\beta_i},H_i)$ the Schr\"odinger representation of the
Heisenberg group $H(\Bbb V_{\beta_i})$,
and choose $T_i(\varepsilon)\in GL_{\Bbb C}(H_i)$ for each 
$\varepsilon\in\Bbb F[\overline\beta_i]^{\times}$ such that 
$$
 \pi_{\overline\beta_i}(\varepsilon^{-1}v\varepsilon,s)
 =T_i(\varepsilon)^{-1}\circ\pi_{\overline\beta_i}(v,s)\circ
  T_i(\varepsilon)
$$
for all $(v,s)\in H(\Bbb V_{\beta_i})$. Let 
$c_i\in Z^2(\Bbb F[\overline\beta_i]^{\times},\Bbb C^{\times})$ be the
2-cocycle associated with $T_i$, that is
$$
 T_i(\varepsilon)\circ T_i(\eta)
 =c_i(\varepsilon,\eta)\cdot T_i(\varepsilon\eta)
$$
for all $\varepsilon, \eta\in\Bbb F[\overline\beta_i]^{\times}$. Now we have a
surjective group homomorphism 
$$
 ((X,s),(v_1,s_1),\cdots,(v_r,s_r))
 \mapsto\left(X+\sum_{i=1}^rv_i,s\cdot\prod_{i=1}^rs_i\right)
$$
of $H(M_{\beta})\times\prod_{i=1}^rH(\Bbb V_{\beta_i})$ onto 
$H(\Bbb V_{\beta})$, and the tensor product 
$\pi_{\overline\beta,M}\otimes\bigotimes_{i=1}^r\pi_{\overline\beta_i}$, 
which is trivial on the
kernel of the group homomorphism, induces the Schr\"odinger
representation of $H(\Bbb V_{\beta})$ on
the space $L^2(M_{\beta}^-)\otimes\bigotimes_{i=1}^rH_i$. Put 
$$
 T(\varepsilon)=T_M(\varepsilon)\otimes\bigotimes_{i=1}^rT_i(\varepsilon_i)
$$ 
for $\varepsilon=\begin{bmatrix}
                  \varepsilon_1&      &         \\
                               &\ddots&         \\
                               &      &\varepsilon_r
                 \end{bmatrix}\in\Bbb F[\overline\beta]^{\times}$, and we have
\begin{equation}
 c_T(\varepsilon,\eta)
 =\prod_{i=1}^rc_i(\varepsilon_i,\eta_i)
\label{eq:reduction-of-c-t-to-basic-case}
\end{equation}
for all $\varepsilon, \eta\in\Bbb F[\overline\beta]^{\times}$. 

Let us assume that the $\Bbb F$-vector subspace $V$ of
$M_n(\Bbb F)$ associated with the $\Bbb F$-linear splitting
$v\mapsto[v]$ contains $M_{\beta}$. Then we have 
$[\varepsilon]\in L_{\beta}$ for all 
$\varepsilon\in\Bbb F[\overline\beta]^{\times}$ and
\begin{equation}
 c_{\overline\beta,\rho}(\varepsilon,\eta)
 =\prod_{i=1}^rc_{\overline\beta_i,\rho}(\varepsilon_i,\eta_i)
\label{eq:reduction-of-c-beta-rho-to-basic-case}
\end{equation}
for all 
$\varepsilon=\begin{bmatrix}
              \varepsilon_1&      &         \\
                           &\ddots&         \\
                           &      &\varepsilon_r
             \end{bmatrix}, 
 \eta=\begin{bmatrix}
              \eta_1&      &         \\
                    &\ddots&         \\
                    &      &\eta_r
             \end{bmatrix}\in\Bbb F[\overline\beta]^{\times}$. 

\begin{prop}\label{prop:shur-multiplier-is-almost-coboundary}
Let a group homomorphism $\widetilde\rho:L_{\beta}\to\Bbb C^{\times}$
be an extension of $\rho:\Bbb F[\overline\beta]\to\Bbb C^{\times}$. Then we
have
$$
 c_{\overline\beta,\rho}(\varepsilon,\eta)
 =\widetilde\rho(2^{-1}\varepsilon[v_{\eta}]\varepsilon^{-1})\cdot
  \widetilde\rho(2^{-1}[v_{\varepsilon\eta}])^{-1}\cdot
  \widetilde\rho(2^{-1}[v_{\varepsilon}])
$$
for all $\varepsilon, \eta\in\Bbb F[\overline\beta]^{\times}$.
\end{prop}
\begin{proof}
The relation \eqref{eq:multiplication-formula-of-v-epsilon} gives 
$\varepsilon^{-1}v_{\varepsilon}\varepsilon=-v_{\varepsilon^{-1}}$ so
that we have
\begin{align*}
 c_{\overline\beta,\rho}(\varepsilon,\eta)
 &=\widehat\tau\left(2^{-1}\langle v_{\varepsilon},
                \varepsilon v_{\eta}\varepsilon\rangle_{\overline\beta}\right)
  =\widehat\tau\left(2^{-1}\langle v_{\eta},v_{\varepsilon^{-1}}\rangle_{\overline\beta}
              \right)\\
 &=\rho\left(2^{-1}\gamma(v_{\eta},\varepsilon^{-1})\right)\\
 &=\widetilde\rho\left(2^{-1}\{
     \varepsilon[v_{\eta}]\varepsilon^{-1}
       -[v_{\varepsilon\eta}-v_{\varepsilon}]\}\right).
\end{align*}
\end{proof}

\begin{thm}\label{th:hypothesis-is-valid-for-separable-chara-poly-case}
The hypothesis \ref{hypo:tirviality-of-schur-multiplier} is valid 
if the characteristic polynomial of $\overline\beta\in M_n(\Bbb F)$ is
separable.
\end{thm}
\begin{proof}
The formula \eqref{eq:reduction-of-c-t-to-basic-case} and 
\eqref{eq:reduction-of-c-beta-rho-to-basic-case} enable us to 
assume that
the characteristic polynomial of $\overline\beta\in M_n(\Bbb F)$ is irreducible
over $\Bbb F$. In this case $\Bbb F[\overline\beta]$ is a finite field so that
the multiplicative group $\Bbb F[\overline\beta]^{\times}$ is a cyclic
group. Then $H^2(\Bbb F[\overline\beta]^{\times},\Bbb C^{\times})$ is
well-known to be trivial. 
\end{proof}

Due to the formula \eqref{eq:reduction-of-c-t-to-basic-case} and 
\eqref{eq:reduction-of-c-beta-rho-to-basic-case}, the hypothesis 
\ref{hypo:tirviality-of-schur-multiplier} is reduced to the case 
$\chi_{\overline\beta}(t)=p(t)^e$ with an irreducible polynomial 
$p(t)\in\Bbb F[t]$. In the following subsections, we will consider 
two extreme cases 
\begin{enumerate}
\item $e=1$ (in subsections 
  \ref{subsec:irreducible-characteristic-polynomial-case} and 
\ref{subsec:quadratic-cuspidal-case}), and 
\item $\deg p(t)=1$ (in subsection 
\ref{subsec:splitting-characteristic-polynomial-case})
\end{enumerate}
to see how the $2$-cocycles $c_{\overline\beta,\rho}$ and $c_T$ are written as
coboundaries. 

\subsection{}
\label{subsec:irreducible-characteristic-polynomial-case}
In this subsection, we will assume that the characteristic polynomial 
$\chi_{\overline\beta}(t)$ of $\overline\beta$ is irreducible over $\Bbb F$. Then, as is
pointed out in the proof of Theorem
\ref{th:hypothesis-is-valid-for-separable-chara-poly-case}, the group
of the Schur multipliers $H^2(\Bbb F[\overline\beta]^{\times},\Bbb C^{\times})$
is trivial. In the following subsections, we will try to express the
$2$-cocycles $c_{\overline\beta,\rho}$ and $c_T$ as coboundaries as explicit as
possible. 

There exists a $\Bbb F$-basis $\{u_1,\cdots,u_n\}$ 
of $\Bbb F[\overline\beta]$
such that the regular representation of the field $\Bbb F[\overline\beta]$ with
respect to the $\Bbb F$-basis is the inclusion mapping of 
$\Bbb F[\overline\beta]$ into $M_n(\Bbb F)$. So the trace of $\Bbb F[\overline\beta]$
over $\Bbb F$ is $T_{\Bbb F[\overline\beta]/\Bbb F}(x)=\text{\rm tr}(x)$ for
all $x\in\Bbb F[\overline\beta]$. Then we have a group isomorphism
$a\mapsto\rho_a$ of $\Bbb F[\overline\beta]$ onto the group of additive
character of $\Bbb F[\overline\beta]$ defined by 
$\rho_a(x)=\widehat\tau\left(\text{\rm tr}(ax)\right)$. So 
$\rho=\rho_a$ has a extension $\widetilde\rho(x)=\text{\rm tr}(ax)$ to 
$L_{\beta}=M_n(\Bbb F)$ which is $\Bbb F[\overline\beta]^{\times}$-invariant. 
Then 
Proposition \ref{prop:shur-multiplier-is-almost-coboundary} shows that
the 2-cocycle 
$c_{\overline\beta,\rho}\in Z^2(\Bbb F[\overline\beta]^{\times},\Bbb C^{\times})$ is a
coboundary
$$
 c_{\overline\beta,\rho}(\varepsilon,\eta)
 =\gamma(\eta)\cdot\gamma(\varepsilon\eta)^{-1}
  \gamma(\varepsilon)
$$
where 
$\gamma(\varepsilon)
 =\widehat\tau\left(2^{-1}\text{\rm tr}(a[v_{\varepsilon}])\right)$ for 
$\varepsilon\in\Bbb F[\overline\beta]^{\times}$. 

We will apply the general setting of subsection 
\ref{subsec:schrodinger-rep-associated-with-polarization} 
to give an explicit description of 2-cocycle 
$c_T\in Z^2(\Bbb F[\overline\beta]^{\times},\Bbb C^{\times})$ by defining a
canonical polarization of $\Bbb V_{\beta}$. 
To begin with there exists a symmetric $g\in GL_n(\Bbb F)$ such that 
$^t\overline\beta=g\overline\beta g^{-1}$ (for example 
$g=\left(\text{\rm tr}(u_iu_j)\right)_{i,j=1,\cdots,n}$). 
Then the $\Bbb F$-linear endomorphism 
$X\mapsto X^{\ast}=g^{-1}\,^tXg$ is an involution of $M_n(\Bbb F)$,
that is, $(X^{\ast})^{\ast}=X$ for all $X\in M_n(\Bbb F)$. 

\begin{rem}
\label{remark:polarization-is-symmetric-anti-symmetric-decompo}
We have canonical $\Bbb F$-linear isomorphisms
\begin{equation*}
 \Bbb F[\overline\beta]{\otimes}_{\Bbb F}\Bbb F[\overline\beta]\,\tilde{\to}\,
 \text{\rm End}_{\Bbb F}(\Bbb F[\overline\beta])\,\tilde{\to}\,
 M_n(\Bbb F)
%\label{eq:canonical-isomorphism-of-tensor-to-matrix}
\end{equation*}
where the first isomorphism is given by $a\otimes b\mapsto f_{a,b}$
with $f_{a,b}(x)=\text{\rm tr}(bx)\cdot a$ for $x\in\Bbb F[\overline\beta]$,
and the second isomorphism is to take the representation matrix with
respect to $\{u_1,\cdots,u_n\}$. Then the involution 
$X\mapsto X^{\ast}=g^{-1}\,^tXg$ of $M_n(\Bbb F)$ induces the 
$\Bbb F$-linear endomorphism $a\otimes b\mapsto b\otimes a$ of 
$\Bbb F[\overline\beta]{\otimes}_{\Bbb F}\Bbb F[\overline\beta]$.
\end{rem}

Put
$$
 W_{\pm}=\{X\in M_n(\Bbb F)\mid X^{\ast}=\pm X\}.
$$
Then $M_n(\Bbb F)=W_-\oplus W_+$ and $\Bbb F[\overline\beta]\subset W_+$. Let
us denote by $\Bbb W_{\pm}$ the image of $W_{\pm}$ by the canonical
surjection $M_n(\Bbb F)\to\Bbb V_{\beta}$. Then 
$\Bbb V_{\beta}=\Bbb W_-\oplus\Bbb W_+$ is a polarization of the
symplectic $\Bbb F$-space $\Bbb V_{\beta}$ because we have 
$\text{\rm tr}(XY\overline\beta)=\text{\rm tr}(Y^{\ast}X^{\ast}\overline\beta)$ for all
$X, Y\in M_n(\Bbb F)$. Using this polarization, we have 
$$
 c_T(\varepsilon,\eta)
 =\gamma_{\widehat\tau}(Q_{\varepsilon,\eta})^{-1}
$$
for $\varepsilon, \eta\in\Bbb F[\overline\beta]^{\times}$ where 
$\gamma_{\widehat\tau}(Q_{\varepsilon,\eta})$ is the generalized Weil
constant of the $\Bbb F$-quadratic form defined on 
$\Bbb W_-\times\Bbb W_-\times\Bbb W_-$ by 
$$
 Q_{\varepsilon,\eta}(x,y,z)
 =\langle x,y\sigma_{\eta}\rangle_{\overline\beta}
 +\langle y,z\sigma_{\varepsilon}\rangle_{\overline\beta}
 -\langle x,z\sigma_{\varepsilon\eta}\rangle_{\overline\beta}.
$$
Our problem is to find an explicit function 
$\delta:\Bbb F[\overline\beta]^{\times}\to\Bbb C^{\times}$ such that
$$
 c_T(\varepsilon,\eta)
 =\delta(\eta)\cdot\delta(\varepsilon\eta)^{-1}\delta(\varepsilon)
$$
for all $\varepsilon, \eta\in\Bbb F[\overline\beta]^{\times}$. 
As we will see in the next subsection, 
the solution is quite interesting even in the simplest case $n=2$.

\subsection{}\label{subsec:quadratic-cuspidal-case}
Suppose that the characteristic polynomial of $\overline\beta\in M_2(\Bbb F)$
is irreducible over $\Bbb F$, and use the notations of the preceding
subsection. We can assume that $\overline\beta=\begin{bmatrix}
                                       0&\alpha\\
                                       1&0
                                      \end{bmatrix}$ where 
$\alpha\in\Bbb F^{\times}$ is not a square. Then we have 
$^t\overline\beta=g\overline\beta g^{-1}$ with $g=\begin{bmatrix}
                                 0&1\\
                                 1&0
                                \end{bmatrix}$ and 
$$
 W_+=\left\{\begin{bmatrix}
             a&b\\
             c&a
            \end{bmatrix}\in M_2(\Bbb F)\right\},
 \quad
 W_-=\left\{\begin{bmatrix}
             a&0\\
             0&-a
            \end{bmatrix}\in M_2(\Bbb F)\right\}.
$$
We will identify $\Bbb W_+$  and $\Bbb W_-$ with $\Bbb F$ by means of
the mappings 
$$
 y\mapsto\begin{bmatrix}
          0&y\\
          0&0
         \end{bmatrix}\pmod{\Bbb F[\overline\beta]},
 \quad
 x\mapsto\begin{bmatrix}
          x&0\\
          0&-x
         \end{bmatrix}\pmod{\Bbb F[\overline\beta]}
$$
respectively. Then $Sp(\Bbb V_{\beta})=SL_2(\Bbb F)$ by the block
description with respect to the polarization 
$\Bbb V_{\beta}=\Bbb W_-\oplus\Bbb W_+$. 
For $\zeta=\begin{bmatrix}
            a&\alpha b\\
            b&a
           \end{bmatrix}\in\Bbb F[\overline\beta]$, put 
$$
 \overline\zeta=\begin{bmatrix}
                 a&-\alpha b\\
                -b&a
                \end{bmatrix}\in\Bbb F[\overline\beta],
 \quad
 \zeta_+=a,
 \quad
 \zeta_-=b
$$
(a collision of notation with $\overline\beta=\beta\npmod{\frak p}$
occurs here, but there may be no risk of confusion in the following
arguments). 
Then for $\varepsilon\in\Bbb F[\overline\beta]^{\times}$, we have
$$
 \sigma_{\varepsilon}
 =\begin{pmatrix}
   (\varepsilon/\overline\varepsilon)_+
           &2\alpha(\varepsilon/\overline\varepsilon)_-\\
   2^{-1}(\varepsilon/\overline\varepsilon)_-
           &(\varepsilon/\overline\varepsilon)_+
  \end{pmatrix}
 \in SL_2(\Bbb F).
$$

\begin{prop}\label{prop:explicit-formula-of-2-cocycle-in-quadratic-case}
For $\varepsilon, \eta\in\Bbb F[\overline\beta]^{\times}$, we have
\begin{enumerate}
\item if $(\varepsilon/\overline\varepsilon)_-\cdot (\eta/\overline\eta)_-
          \cdot(\varepsilon\eta/\overline{\varepsilon\eta})_-=0$, then 
      $c_{\widehat\tau}(\varepsilon,\eta)=1$,
\item if $(\varepsilon/\overline\varepsilon)_-\cdot (\eta/\overline\eta)_-
          \cdot(\varepsilon\eta/\overline{\varepsilon\eta})_-\neq 0$,
          then 
$$
 c_{\widehat\tau}(\varepsilon,\eta)
 =\gamma_{\widehat\tau}(\alpha(\varepsilon/\overline\varepsilon)_-)\cdot
  \gamma_{\widehat\tau}(\alpha(\eta/\overline\eta)_-)\cdot
  \gamma_{\widehat\tau}(\alpha(\varepsilon\eta/\overline{\varepsilon\eta})_-)^{-1}.
$$
\end{enumerate}
\end{prop}
\begin{proof}
The quadratic form $Q_{\varepsilon,\eta}$ on 
$\Bbb W_-\times\Bbb W_-\times\Bbb W_-=\Bbb F\times\Bbb F\times\Bbb F$
is 
$$
 Q_{\varepsilon,\eta}(x,y,z)
 =4\alpha(\eta/\overline\eta)_-\cdot xy
  +4\alpha(\varepsilon/\overline\varepsilon)_-\cdot yz
  -4\alpha(\varepsilon\eta/\overline{\varepsilon\eta})_-\cdot zx
$$
and the symmetric matrix associated with it is
$$
 \widetilde Q_{\varepsilon,\eta}
 =\begin{bmatrix}
   0&4\alpha(\eta/\overline\eta)_-
      &-4\alpha(\varepsilon\eta/\overline{\varepsilon\eta})_-\\
   4\alpha(\eta/\overline\eta)_-&0
      &4\alpha(\varepsilon/\overline\varepsilon)_-\\
   -4\alpha(\varepsilon\eta/\overline{\varepsilon\eta})_-
      &4\alpha(\varepsilon/\overline\varepsilon)_-&0
  \end{bmatrix}.
$$
If $(\varepsilon/\overline\varepsilon)_-=0$, then 
$\delta=\varepsilon/\overline\varepsilon=\pm 1$ and 
$$
 P\widetilde Q_{\varepsilon,\eta}\,^tP
 =\begin{bmatrix}
   0&8\delta\alpha(\eta/\overline\eta)_-&0\\
   8\delta\alpha(\eta/\overline\eta)_-&0&0\\
   0&0&0
  \end{bmatrix}
$$
with $P=\begin{bmatrix}
         1&0&0\\
         0&\delta&-1\\
         0&\delta&1
        \end{bmatrix}$. Hence the quadratic form
$Q_{\varepsilon,\eta}$ is equivalent to $0$ or the orthogonal sum of
$0$ and a hyperbolic plane, and we have
$$
 c_{\widehat\tau}(\varepsilon,\eta)=\gamma_{\widehat\tau}(Q_{\varepsilon,\eta})^{-1}
 =1.
$$
The cases $(\eta/\overline\eta)_-=0$ or 
$(\varepsilon\eta/\overline{\varepsilon\eta})_-=0$ is treated
similarly. 

If 
$(\varepsilon/\overline\varepsilon)_-(\eta/\overline\eta)_-
 (\varepsilon\eta/\overline{\varepsilon\eta})_-\neq 0$, then
$Q_{\varepsilon,\eta}$ is a regular quadratic form on a finite field
$\Bbb F$ of odd characteristic. Then $Q_{\varepsilon,\eta}$ is
equivalent to the quadratic form
$$
 (x,y,z)\mapsto 
   4\alpha(\varepsilon/\overline\varepsilon)_-\cdot x^2
  +4\alpha(\eta/\overline\eta)_-\cdot y^2
  -4\alpha(\varepsilon\eta/\overline{\varepsilon\eta})_-\cdot z^2
$$
due to \cite[p.39, Th.3.8]{Scharlau}. Hence we have
\begin{align*}
 c_{\widehat\tau}(\varepsilon,\eta)
 &=\gamma_{\widehat\tau}(Q_{\varepsilon,\eta})\\
 &=\gamma_{\widehat\tau}(4\alpha(\varepsilon/\overline\varepsilon)_-)
   \gamma_{\widehat\tau}(4\alpha(\eta/\overline\eta)_-)
   \gamma_{\widehat\tau}(4\alpha(\varepsilon\eta/\overline{\varepsilon\eta})_-)^{-1}.
\end{align*}
\end{proof}

Finally we have an explicit expression of $c_{\widehat\tau}$ as a coboundary

\begin{thm}\rm\label{th:explict-expression-of-c-tau-as-coboundary}
Put
$$
 \delta(\varepsilon)
 =\begin{cases}
   \gamma_{\widehat\tau}(\alpha(\varepsilon/\overline\varepsilon)_-)
     &:(\varepsilon/\overline\varepsilon)_-\neq 0,\\
   \left(\frac{\varepsilon/\overline\varepsilon}
              {\Bbb F}\right)
     &:(\varepsilon/\overline\varepsilon)_-=0
  \end{cases}
$$
for $\varepsilon\in\Bbb F[\overline\beta]^{\times}$. Then we have
$$
 c_{\widehat\tau}(\varepsilon,\eta)
 =\delta(\eta)\delta(\varepsilon\eta)^{-1}\delta(\varepsilon)
$$
for all $\varepsilon,\eta\in\Bbb F[\overline\beta]^{\times}$.
\end{thm}

\subsection{}
\label{subsec:splitting-characteristic-polynomial-case}
In this subsection, we will assume that the characteristic polynomial
of $\overline\beta\in M_n(\Bbb F)$ is $\chi_{\overline\beta}(t)=(t-a)^n$ with
$a\in\Bbb F$. In this case we can suppose that
$$
 \overline\beta=J_n(a)
 =\begin{bmatrix}
   a&1&      &      \\
    &a&\ddots&      \\
    & &\ddots&1   \\
    & &      &a
  \end{bmatrix}
$$
is a Jordan block. Let $W_+$ and $W_-$ be the $\Bbb F$-vector subspace
of $M_n(\Bbb F)$ consisting of the upper triangular matrices and the
lower triangular matrices with the diagonal elements $0$ 
respectively. Let $\Bbb W_{\pm}$ be the image of $W_{\pm}$ with
respect to the canonical surjection $M_n(\Bbb F)\to\Bbb
V_{\beta}$. Then $\Bbb V_{\beta}=\Bbb W_-\oplus\Bbb W_+$ is a
polarization of the symplectic $\Bbb F$-space $\Bbb V_{\beta}$. 
Note that $\Bbb W_+\sigma_{\varepsilon}=\Bbb W_+$ for any 
$\varepsilon\in\Bbb F[\overline\beta]^{\times}$. Hence 
$c_T(\varepsilon,\eta)=1$ for all 
$\varepsilon, \eta\in\Bbb F[\overline\beta]^{\times}$ as described in
subsection
\ref{subsec:schrodinger-rep-associated-with-polarization}. 

Until the end of this subsection 
we will assume that the characteristic of $\Bbb F$ is
greater than $n$. Then we have a group isomorphism 
$$
 \Bbb F^{\times}\times\Bbb F^{n-1}\,\tilde{\to}\,
 \Bbb F[\overline\beta]^{\times}
$$
defined by 
$(r,s_1,\cdots,s_{n-1})\mapsto 
 r\cdot\exp\left(\sum_{k=1}^{n-1}s_kJ_n(0)^k\right)$. Here
$$
 \exp S=\sum_{k=0}^{n-1}\frac 1{k!}S^k
$$
is the exponential of upper triangular matrix $S\in M_n(\Bbb F)$ with
diagonal elements $0$. 
For any $A\in\Bbb F[\overline\beta]$, put $\rho_A(X)=\text{\rm tr}(X\,^tA)$ 
($X\in\Bbb F[\overline\beta]$). Then $A\mapsto\rho_A$ gives a group isomorphism
of $\Bbb F[\overline\beta]$ onto the group of the additive characters of 
$\Bbb F[\overline\beta]$. Let $V$ be the $\Bbb F$-vector space consisting of
the $X\in M_n(\Bbb F)$ such that $\text{\rm tr}(X\,^tA)=0$ for all
$A\in\Bbb F[\overline\beta]$. Then we have $M_n(\Bbb F)=V\oplus\Bbb F[\overline\beta]$
because $\text{\rm ch}\,\Bbb F>n$, and a $\Bbb F$-linear splitting
of the exact sequence \eqref{eq:canonical-exact-seq-of-v-beta} is
defined with respect to this $V$. Now any additive character 
$\rho=\rho_A$ of $\Bbb F[\overline\beta]$ ($A\in\Bbb F[\overline\beta]$) has an
extension $\widetilde\rho(X)=\text{\rm tr}(X\,^tA)$ to 
$L_{\beta}=M_n(\Bbb F)$, and Proposition
\ref{prop:shur-multiplier-is-almost-coboundary} gives
$$
 c_{\overline\beta,\rho}(\varepsilon,\eta)
 =\widehat\tau\left(2^{-1}\text{\rm tr}(
   \varepsilon[v_{\eta}]\varepsilon^{-1}\,^tA)\right)
$$
for all $\varepsilon ,\eta\in\Bbb F[\overline\beta]^{\times}$. 

\begin{prop}\label{prop:invariant-extension-of-additive-character}
An additive character $\rho=\rho_A$ of $\Bbb F[\overline\beta]$ with 
$A\in\Bbb F[\overline\beta]$ has an extension $\widetilde\rho$ to $M_n(\Bbb F)$
as an additive character such that 
$\widetilde\rho(\varepsilon X\varepsilon^{-1})=\widetilde\rho(X)$ for
all $\varepsilon\in\Bbb F[\overline\beta]^{\times}$ if and only if 
$A\in\Bbb F[\overline\beta]$ is a diagonal matrix. In this case we have 
$c_{\overline\beta,\rho}(\varepsilon,\eta)=1$ for all 
$\varepsilon, \eta\in\Bbb F[\overline\beta]^{\times}$.
\end{prop}
\begin{proof}
Every extensions $\widetilde\rho$ 
of $\rho$ to $M_n(\Bbb F)$ is written in the form
$$
 \widetilde\rho(X)
 =\widehat\tau\left(\text{\rm tr}(X\,^tA)\right)\cdot
  \widehat\tau\left(\langle X,X_0\rangle_{\overline\beta}\right)
$$
with certain $X_0\in M_n(\Bbb F)$. So 
$\widetilde\rho(\varepsilon X\varepsilon^{-1})=\widetilde\rho(X)$ for
all $\varepsilon\in\Bbb F[\overline\beta]^{\times}$ means that 
$B=\,^tA+X_0\overline\beta-\overline\beta X_0$ is commutative with any 
$\varepsilon\in\Bbb F[\overline\beta]^{\times}$. Since 
$c1_n-\overline\beta\in\Bbb F[\overline\beta]^{\times}$ for a suitable $c\in\Bbb F$, 
the matrix $^tA+X_0\overline\beta-\overline\beta X_0$ is commutative
with $\overline\beta$. Hence  
$B\in\Bbb F[\overline\beta]$ and 
$$
 \rho(X)=\widetilde\rho(X)=\widehat\tau\left(\text{\rm tr}(XB)\right)
$$
for all $X\in\Bbb F[\overline\beta]$. This means that $A$ is a diagonal
matrix. In this case we have
$$
 c_{\overline\beta,\rho}(\varepsilon,\eta)
 =\widehat\tau\left(2^{-1}\text{\rm tr}([v_{\eta}]\,^tA)\right)=1
$$
for all $\varepsilon, \eta\in\Bbb F[\overline\beta]^{\times}$.
\end{proof}

We can calculate the 2-cocycle $c_{\overline\beta,\rho}$ explicitly for small
$n$. Take an additive character $\rho=\rho_A$ of $\Bbb F[\overline\beta]$ with
$$
 A=\sum_{k=0}^{n-1}\rho_k J_n(0)^k\in\Bbb F[\overline\beta].
$$
For the calculation of $c_{\overline\beta,\rho}(\varepsilon,\eta)$, there is no
loss of generality if we assume that $\varepsilon$ and $\eta$ are of
the form $\exp S$ with an upper triangular
$S\in\Bbb F[\overline\beta]\subset M_n(\Bbb F)$ with diagonal elements $0$. 
In the following
examples, we will give explicit formulas for $c_{\overline\beta,\rho}$ in the
cases $n=2,3,4$.

\begin{ex}\label{ex:trivial-2-cocycle-in-n-is-2}
For $\varepsilon=\exp(\begin{bmatrix}
                       0&r\\
                       0&0
                      \end{bmatrix}), \eta=\exp(\begin{bmatrix}
                                                 0&u\\
                                                 0&0
                                                \end{bmatrix})
\in\Bbb F[\overline\beta]^{\times}$, we have
$$
 c_{\overline\beta,\eta}(\varepsilon,\eta)
 =\widehat\tau\left(2^{-1}\rho_1^2\cdot(r^2u+ru^2)\right).
$$
If $\text{\rm ch}\,\Bbb F>3$, then put
$$
 \delta(\varepsilon)
 =\widehat\tau\left(-\frac 1{2\cdot 3}\rho_1^2\cdot r^3\right)
$$
and we have
$$
 c_{\overline\beta,\eta}(\varepsilon,\eta)
 =\delta(\eta)\delta(\varepsilon\eta)^{-1}\delta(\varepsilon)
$$
for all $\varepsilon,\eta\in\Bbb F[\overline\beta]^{\times}$. 
\end{ex}

\begin{ex}\label{ex:trivial-2-cocycle-in-n-is-3}
If $\text{\rm ch}\,\Bbb F>5$, put
$$
 \delta(\varepsilon)
 =\widehat\tau\left(-\frac 12\left\{
   \frac 13\rho_1^2\cdot r^3+2\rho_1\rho_2\cdot r^2s
    +\rho_2^2\left(rs^2-\frac 1{2^2\cdot 5}r^5\right)\right\}\right)
$$
for $\varepsilon=\exp(\begin{bmatrix}
                       0&r&s\\
                       0&0&r\\
                       0&0&0
                      \end{bmatrix})\in\Bbb F[\overline\beta]^{\times}$. Then
we have
$$
 c_{\overline\beta,\rho}(\varepsilon,\eta)
 =\delta(\eta)\delta(\varepsilon\eta)^{-1}\delta(\varepsilon)
$$
for all $\varepsilon, \eta\in\Bbb F[\overline\beta]^{\times}$.
\end{ex}

\begin{ex}\label{ex:trivial-2-cocycle-in-n-is-4}
If $\text{\rm ch}\,\Bbb F>7$, put
$$
 \delta(\varepsilon)
 =\widehat\tau\left[-\frac 12\left\{
   \begin{array}{l}
    \frac 13\rho_1^2\cdot r^3+2\rho_1\rho_2\cdot r^2s
     +\rho_2^2\cdot\left(2rs^2+r^2t-\frac 1{2\cdot 3\cdot 5}r^5\right)\\
   +\rho_1\rho_3\cdot\left(2rs^2+2r^2t+\frac 1{2\cdot 3\cdot 5}r^5\right)\\
   +\rho_2\rho_3\cdot\left(4rst+\frac 43s^3-\frac 13r^4s\right)\\
   +\rho_3^2\cdot\left(s^2t+rt^2
   +\frac 1{2^2\cdot 3}r^4t-\frac 13r^3s^2+\frac 1{2^2\cdot 3^2\cdot 7}r^7\right)
   \end{array}\right\}\right]
$$
for $\varepsilon=\exp(\begin{bmatrix}
                       0&r&s&t\\
                       0&0&r&s\\
                       0&0&0&r\\
                       0&0&0&0
                      \end{bmatrix})\in\Bbb F[\overline\beta]^{\times}$. Then
we have
$$
 c_{\overline\beta,\rho}(\varepsilon,\eta)
 =\delta(\eta)\delta(\varepsilon\eta)^{-1}\delta(\varepsilon)
$$
for all $\varepsilon ,\eta\in\Bbb F[\overline\beta]^{\times}$. 
\end{ex}

These examples strongly suggest that the following conjecture is valid for
all $n\geq 2$. 

\begin{conj}\label{conj:schur-multiplier-is-trivial-in-splitting-case}
The Schur multiplier 
$[c_{\overline\beta,\rho}]\in H^2(\Bbb F[\overline\beta]^{\times},\Bbb C^{\times})$ is
trivial for all additive characters $\rho$ of $\Bbb F[\overline\beta]$ 
if the characteristic of $\Bbb F$ is big enough.
\end{conj}

\begin{rem}\label{remark:problem-to-find-non-trivial-shur-multiplier}
It may be an interesting problem to find a regular $\overline\beta\in M_n(\Bbb F)$
whose characteristic polynomial is a power
of an irreducible polynomial over $\Bbb F$ of degree greater than 
or equal to $2$
and an additive character $\rho$ of $\Bbb F[\overline\beta]$ such that 
the Schur multiplier 
$[c_{\overline\beta,\rho}]\in H^2(\Bbb F[\overline\beta]^{\times},\Bbb C^{\times})$ is
not trivial.
\end{rem}

%% file: weilrep.tex
\section{Construction of the regular characters via Weil representation}
\label{sec:regular-character-via-weil-rep-over-finite-field}
In this section we will construct the irreducible representation 
$\pi_{\psi}$ of subsection \ref{sec:schur-multiplier} by means of
Schr\"odinber representation of Heisenberg group and Weil
representation over finite field. We will also describe the Schur
multiplier of Hypothesis
\ref{hyp:schur-multiplier-c-psi-is-trivial-for-all-psi} by means of
the Schur multiplier defined in Proposition
\ref{prop:schur-multiplier-associated-with-finite-symplectic-space}. 

Throughout this section the finite
field $\Bbb F$ is supposed to be of odd characteristic, although the
arguments of subsection \ref{sec:structure-of-g-l-1} work well without
this assumption. 

\subsection{}
\label{sec:structure-of-g-l-1}
Let us determine the 2-cocycle of the group extension
\begin{equation}
 0\to M_n(O_{l-1})\to K_{l-1}\to M_n(\Bbb F)\to 0
\label{eq:k-l-is-a-group-extension}
\end{equation}
defined by the isomorphism $K_l\,\tilde{\to}\,M_n(O_{l-1})$ 
($1_n+\varpi^lX\npmod{\frak{p}^r}\mapsto
  X\npmod{\frak{p}^{l-1}}$) and 
$K_{l-1}/K_l\,\tilde{\to}\,M_n(\Bbb F)$ induced by 
$1_n+\varpi^{l-1}X\npmod{\frak{p}^r}
 \mapsto X\npmod{\frak p}$. Fix a mapping 
$\lambda:M_n(\Bbb F)\to M_n(O)$ such that $X=\lambda(X)\npmod{\frak p}$
for all $X\in M_n(\Bbb F)$ and $\lambda(0)=0$, 
and define a mapping $l:M_n(\Bbb F)\to K_{l-1}$ by 
$X\mapsto 1_n+\varpi^{l-1}\lambda(X)\npmod{\frak{p}^r}$. Then, for
any $k=1_n+\varpi^lS\npmod{\frak{p}^r}\in K_l$, we have
\begin{align*}
 &l(X)kl(X)^{-1}\\
 &=(1_n+\varpi^{l-1}\lambda(X))(1_n+\varpi^lS)
        (1_n-\varpi^{l-1}\lambda(X)
            +\varpi^{2(l-1)}\lambda(X)^2)\nnpmod{\frak{p}^r}\\
 &=1_n+\varpi^lS\nnpmod{\frak{p}^r},
\end{align*}
and
$$
 l(X)l(Y)l(X+Y)^{-1}
 =1_n+\varpi^l\left(\varpi^{l-2}\lambda(X)\lambda(Y)+\mu(X,Y)\right)
  \nnpmod{\frak{p}^r}
$$
for all $X,Y\in M_n(\Bbb F)$, where 
$\mu:M_n(\Bbb F)\times M_n(\Bbb F)\to M_n(O)$ is defined by
$$
 \lambda(X)+\lambda(Y)-\lambda(X+Y)=\varpi\mu(X,Y).
$$
So the 2-cocycle of the group extension
\eqref{eq:k-l-is-a-group-extension} is
$$
 [(\overline X,\widehat Y)\mapsto
  \varpi^{l-2}XY+\mu(\overline X,\widehat Y)\nnpmod{\frak{p}^{l-1}}]
 \in Z^2(M_n(\Bbb F),M_n(O_{l-1}))
$$
with the trivial action of $M_n(\Bbb F)$ on $M_n(O_{l-1})$, where 
$\overline X=X\npmod{\frak p}\in M_n(\Bbb F)$. Now we have two
2-cocycles
\begin{align*}
 c&=[(\overline X,\widehat Y)\mapsto
       \varpi^{l-2}XY\nnpmod{\frak{p}^{l-1}}]
    \in Z^2(M_n(\Bbb F),M_n(O_{l-1}))\;\text{\rm and}\\
 \mu&=[(X,Y)\mapsto\mu(X,Y)\nnpmod{\frak{p}^{l-1}}]
    \in Z^2(M_n(\Bbb F),M_n(O_{l-1})).
\end{align*}
Let us denote by $\Bbb G$ and $\Bbb M$ the groups associated with the
2-cocycles $c$ and $\mu$ respectively. More precisely, the group operation on 
$\Bbb G=M_n(\Bbb F)\times M_n(O_{l-1})$ is defined by
$$
 (\overline X,\overline S)\cdot(\widehat Y,\overline T)
 =((X+Y)\wsphat,S+T+\varpi^{l-2}XY\nnpmod{\frak{p}^{l-1}})
$$
and the group operation on $\Bbb M=M_n(\Bbb F)\times M_n(O_{l-1})$ is
defined by
$$
 (X,\overline S)\cdot(Y,\overline T)
 =(X+Y,S+T+\mu(X,Y)\nnpmod{\frak{p}^{l-1}}).
$$
The fiber product of $\Bbb G$ and $\Bbb M$ with respect to the
projections $\Bbb G\to M_n(\Bbb F)$ and $\Bbb M\to M_n(\Bbb F)$ is
denoted by $\Bbb G\times_{M_n(\Bbb F)}\Bbb M$. That is 
$\Bbb G\times_{M_n(\Bbb F)}\Bbb M$ is the subgroup of
the direct product $\Bbb G\times\Bbb M$ consisting of the elements 
$$
 (X;S,S^{\prime})=((X,S),(X,S^{\prime})).
$$
Then we have a surjective group homomorphism 
$\Bbb G\times_{M_n(\Bbb F)}\Bbb M\to K_{l-1}$ 
\begin{equation}
 (X;\overline S,\overline{S^{\prime}})
  \mapsto
  1_n+\varpi^{l-1}\lambda(X)+\varpi^l(S+S^{\prime})\nnpmod{\frak{p}^r})
\label{eq:surjection-of-fiber-product-to-k-l-1}
\end{equation}
whose kernel 
$\langle\Bbb G,\Bbb M\rangle$ consists of the elements $(0;S,-S)$. 
A morphism 
$$
 \text{\rm tr}_{\beta}:M_n(O_{l-1})\to O_{l-1}
 \quad
 (X\nnpmod{\frak{p}^{l-1}}\mapsto
  \text{\rm tr}(X\beta)\nnpmod{\frak{p}^{l-1}})
$$
of additive groups induces a morphism 
$$
 \text{\rm tr}^{\ast}_{\beta}:H^2(M_n(\Bbb F),M_n(O_{l-1}))
  \to H^2(M_n(\Bbb F),O_{l-1})
$$
of cohomology groups. Here $M_n(\Bbb F)$ acts on $O_{l-1}$ trivially. 
The group associated with the 2-cocycle 
$c_{\beta}=c\circ\text{\rm tr}_{\beta}\in Z^2(M_n(\Bbb F),O_{l-1})$ is 
$\mathcal{H}_{\beta}=M_n(\Bbb F)\times O_{l-1}$ with group operation
$$
 (\overline X,s)\cdot(\widehat Y,t)
 =((X+Y)\wsphat,s+t+\varpi^{l-2}\text{\rm tr}(XY\beta)
                                            \nnpmod{\frak{p}^{l-1}}).
$$
We have a surjective group homomorphism
\begin{equation}
 \Bbb G\times_{M_n(\Bbb F)}\Bbb M\to\mathcal{H}_{\beta}
 \qquad
 ((X;\overline S,\overline{S^{\prime}})\mapsto
  (X,\overline{\text{\rm tr}(S\beta)}))
\label{eq:surjection-of-fiber-product-to-mathcal-h-beta}
\end{equation}
whose kernel is 
\begin{equation}
 \{(0;\overline S,\overline{S^{\prime}})\mid
    \text{\rm tr}(S\beta)\equiv 0\nnpmod{\frak{p}^{l-1}}\}.
\label{eq:kernel-of-projection-of-fibre-product-to-mathcal-h-beta}
\end{equation}
The center of $\mathcal{H}_{\beta}$ is 
$Z(\mathcal{H}_{\beta})=\Bbb F[\overline\beta]\times O_{l-1}$, in fact 
$(X,s)\in Z(\mathcal{H}_{\beta})$ is equivalent to 
$\langle X,Y\rangle_{\overline\beta}=0$ for all $Y\in M_n(\Bbb F)$,
which is equivalent to $X\in\Bbb F[\overline\beta]$. Let us denote by 
$(\Bbb G\times_{M_n(\Bbb F)}\Bbb M)(\Bbb F[\overline\beta])$ 
he inverse image
of the center $Z(\mathcal{H}_{\beta})$ by the surjection 
\eqref{eq:surjection-of-fiber-product-to-mathcal-h-beta}. Then 
$(\Bbb G\times_{M_n(\Bbb F)}\Bbb M)(\Bbb F[\overline\beta])$ is projected
onto $K_{l-1}(\Bbb F[\overline\beta])$ by the surjection 
\eqref{eq:surjection-of-fiber-product-to-k-l-1}. We have a group
homomorphism
$$
 \psi_0:\Bbb M\to\Bbb C^{\times}
 \qquad
 ((X,\overline{S^{\prime}})\mapsto
  \tau\left(\varpi^{-l}\text{\rm tr}(\lambda(X)
                          +\varpi S^{\prime})\beta\right))
$$
which induces a group homomorphism 
$\widetilde{\psi}_0:\Bbb G\times_{M_n(\Bbb F)}\Bbb M\to\Bbb C^{\times}$
via the projection $\Bbb G\times_{M_n(\Bbb F)}\Bbb M\to\Bbb M$. 

Take
any $\psi\in Y(\psi_{\beta})$. Combining with the projection 
\eqref{eq:surjection-of-fiber-product-to-k-l-1}, we define a group
homomorphism 
$\widetilde\psi:(\Bbb G\times_{M_n(\Bbb F)}\Bbb M)(\Bbb F[\overline\beta])
\to\Bbb C^{\times}$. 
Then $\widetilde\psi_0^{-1}\cdot\widetilde\psi$ induces a group
homomorphism $\rho$ of
$Z(\mathcal{H}_{\beta})$ to $\Bbb C^{\times}$ because 
$\widetilde\psi_0^{-1}\cdot\widetilde\psi$ is trivial on the kernel 
\eqref{eq:kernel-of-projection-of-fibre-product-to-mathcal-h-beta}. 
The inverse image of $K_l$ under the surjection 
\eqref{eq:surjection-of-fiber-product-to-k-l-1} is projected onto 
$O_{l-1}\subset Z(\mathcal{H}_{\beta})$ by the surjection 
\eqref{eq:surjection-of-fiber-product-to-mathcal-h-beta}, 
and $\psi|_{K_l}=\psi_{\beta}$
means that $\rho|_{O_{l-1}}=\psi_1$ where 
$$
 \psi_1:O_{l-1}\to\Bbb C^{\times}
 \qquad
 (s\nnpmod{\frak{p}^{l-1}}\mapsto\tau(\varpi^{-(l-1)}s))
$$
Thus we have a bijection of $Y(\psi_{\beta})$ onto the subset of 
the character group of $Z(\mathcal{H}_{\beta})$ consisting of the
extensions of $\psi_1$.

Let us consider the action of 
$\mathcal{C}=\left(O_r[\beta_r]\right)^{\times}$ on the groups
$\Bbb G\times_{M_n(\Bbb F)}\Bbb M$ and $\mathcal{H}_{\beta}$. For any 
$(\overline X;\overline S,\overline{S^{\prime}})\in
 \Bbb G\times_{M_n(\Bbb F)}\Bbb M$ and 
$\varepsilon\in  O[\beta]^{\times}$, we have
\begin{align*}
 \varepsilon^{-1}(1_n+\varpi^{l-1}\lambda(\overline X)
                    &+\varpi^l(S+S^{\prime}))\varepsilon\\
 &=1_n+\varpi^{l-1}\lambda((\varepsilon^{-1}X\varepsilon)\sphat)
     +\varpi^l(\varepsilon^{-1}(S+S^{\prime})\varepsilon
               +\nu(\overline X,\varepsilon))
\end{align*}
where $\nu:M_n(\Bbb F)\times\mathcal{C}\to M_n(O)$ is defined by
$$
 \varepsilon^{-1}\lambda(X)\varepsilon
 -\lambda(\varepsilon^{-1}X\varepsilon)
 =\varpi\cdot\nu(X,\varepsilon).
$$
Then $\overline\varepsilon\in\mathcal{C}$ acts on 
$(\overline X;\overline S,\overline{S^{\prime}})
 \in\Bbb G\times_{M_n(\Bbb F)}\Bbb M$ by
$$
 (\overline X;\overline S,\overline{S^{\prime}})^{\varepsilon}
 =((\varepsilon^{-1}X\varepsilon)\wsphat;
     \overline{\varepsilon^{-1}S\varepsilon},
      \overline{\varepsilon^{-1}S^{\prime}\varepsilon
                             +\nu(\overline X,\varepsilon)}),
$$
and on 
$(\overline X,s)\in\mathcal{H}_{\beta}$ by 
$(\overline X,s)^{\varepsilon}
 =((\varepsilon^{-1}X\varepsilon)\wsphat,s)$. We have
$$
 \widetilde\psi_0((X;S,S^{\prime})^{\varepsilon})
 =\widetilde\psi_0(X;S,S^{\prime}).
$$

\subsection{}
\label{sec:non-dyadic-case}
Let us modify the group structure on $\mathcal{H}_{\beta}$ into a form
more suitable for further examinations. The coboundary of a mapping
$$
 \delta:M_n(\Bbb F)\to O_{l-1}
 \quad
(X\nnpmod{\frak p}\mapsto 
  2^{-1}\varpi^{l-2}\text{\rm tr}(X^2\beta)\nnpmod{\frak{p}^{l-1}})
$$
is
\begin{align*}
 \partial\delta(\overline X,\widehat Y)
 &=\delta(\widehat Y)-\delta((X+Y)\wsphat)+\delta(\overline X)\\
 &\equiv
  -2^{-1}\varpi^{l-s}\text{\rm tr}((XY+YX)\beta)\nnpmod{\frak{p}^{l-1}}
\end{align*}
so that we will redefine the group operation on 
$\mathcal{H}_{\beta}=M_n(\Bbb F)\times O_{l-1}$ 
to be associated with the $2$-cocycle $c_{\beta}+\partial\beta$, that
is 
$$
 (\overline X,s)\cdot(\widehat Y,t)
 =((X+Y)\wsphat,
   s+t+\overline{2^{-1}\varpi^{l-2}\text{\rm tr}((XY-YX)\beta)}).
$$
Under this group operation, the center $\mathcal{H}_{\beta}$ is 
$Z(\mathcal{H}_{\beta})=\Bbb F[\overline\beta]\times O_{l-1}$, and for
any $(X,s)\in Z(\mathcal{H}_{\beta})$ and 
$(Y,t)\in\mathcal{H}_{\beta}$ we have
$$
 (X,s)\cdot(Y,t)=(X+Y,s+t).
$$
In particular $Z(\mathcal{H}_{\beta})$ is the
direct product of additive groups $\Bbb F[\overline\beta]$ and
$O_{l-1}$. 
The surjection 
\eqref{eq:surjection-of-fiber-product-to-mathcal-h-beta} is now
\begin{equation}
 (\overline X;\overline S,\overline{S^{\prime}})
 \mapsto
 (\overline X,\overline{\text{\rm tr}(S\beta)
                       -2^{-1}\varpi^{l-2}\text{\rm tr}(X^2\beta)}).
\label{eq:new-surjection-of-fiber-product-to-mathcal-h-beta}
\end{equation}
We have a central extension
\begin{equation}
 0\to Z(\mathcal{H}_{\beta})\to\mathcal{H}_{\beta}
  \xrightarrow{(\ast)}\Bbb V_{\beta}\to 0
\label{eq:mathcal-h-beta-is-central-extension}
\end{equation}
where $(\ast):(X,s)\mapsto \dot X=X\npmod{\Bbb F[\overline\beta]}$.
Let us determine the 2-cocycle of this group extension. Fix an 
$\Bbb F$-vector subspace $V_{\beta}\subset M_n(\Bbb F)$ such that 
$M_n(\Bbb F)=V_{\beta}\oplus\Bbb F[\overline\beta]$, and let 
$[x]\in V_{\beta}$ be the representative of $x\in\Bbb V_{\beta}$. Put
$$
 l:\Bbb V_{\beta}\to\mathcal{H}_{\beta}.
 \quad
 (x\mapsto([x],0))
$$
Then for any $X,Y\in\Bbb V_{\beta}$ we have
$$
 l(x)l(y)l(x+y)^{-1}
 =(0,2^{-1}\varpi^{l-2}
     \text{\rm tr}((XY-YX)\beta)\nnpmod{\frak{p}^{l-1}})
$$
where $[x]=\overline{X}, [y]=\overline{Y}$ with $X,Y\in M_n(O)$. 
So the 2-cocycle of the central extension 
\eqref{eq:mathcal-h-beta-is-central-extension} is 
$$
 [(\dot{\overline X},\dot{\widehat Y})\mapsto
   (0,2^{-1}\varpi^{l-2}\text{\rm tr}((XY-YX)\beta)
                          \nnpmod{\frak{p}^{l-1}})]
 \in Z^2(\Bbb V_{\beta},Z(\mathcal{H}_{\beta})).
$$
The group operation on
$\Bbb H_{\beta}=\Bbb V_{\beta}\times Z(\mathcal{H}_{\beta})$ defined 
by this $2$-cocycle is 
$$
 (x,s)\cdot(y,t)
 =(x+y,s+t+
  (0,2^{-1}\varpi^{l-2}\text{\rm tr}((XY-YX)\beta)\nnpmod{\frak{p}^{l-1}}))
$$
where $x=\dot{\overline X}\in\Bbb V_{\beta}$ with $X\in M_n(O)$ 
etc. The group  $\Bbb H_{\beta}$ is isomorphic to 
$\mathcal{H}_{\beta}$ by $(x,(Y,s))\mapsto([x]+Y,s)$. 

Let $\Bbb V_{\beta}=\Bbb W^{\prime}\oplus\Bbb W$ be a polarization of
the symplectic $\Bbb F$-space $\Bbb V_{\beta}$. Then we have defined 
in the subsection
\ref{subsec:schrodinger-rep-associated-with-polarization} 
the Schr\"odinger representation $(\pi_{\beta},L^2(\Bbb W^{\prime}))$ 
of the Heisenberg group $H(\Bbb V_{\beta})$ associated with the
polarization. We have also defined, for each 
$\sigma\in Sp(\Bbb V_{\beta})$, an
element $T_{\widehat\tau}(\sigma)\in GL_{\Bbb C}(L^2(\Bbb W^{\prime}))$ such that 
$$
 \pi_{\beta}(u\sigma,s)
 =T_{\widehat\tau}(\sigma)^{-1}\circ\pi_{\beta}(u,s)\circ 
  T_{\widehat\tau}(\sigma)
$$
for all $(u,s)\in H(\Bbb V_{\beta})$. 

Take a $\psi\in Y(\psi_{\beta})$. Then, as described in the
preceding section, the group homomorphism 
$\widetilde\psi_0^{-1}\cdot\widetilde\psi$ corresponds to a group
homomorphism $\rho\otimes\psi_1$ of 
$Z(\mathcal{H}_{\beta})=\Bbb F[\overline\beta]\times O_{l-1}$ where 
$\rho$ is a group homomorphism of the additive group 
$\Bbb F[\overline\beta]$ to $\Bbb C^{\times}$. More explicitly 
$$
 \rho(\overline X)
 =\tau\left(2^{-1}\varpi^{-1}\text{\rm tr}(X^2\beta)
            -\varpi^{-l}\text{\rm tr}(\lambda(\overline X)\beta)
                                                   \right)
  \cdot\psi(1_n+\varpi^{l-1}\lambda(\overline X))
$$
for $\overline X\in\Bbb F[\overline\beta]$, or
\begin{equation}
 \psi(\overline g)
 =\rho(\overline X)\cdot
  \tau\left(\varpi^{-l}\text{\rm tr}(X\beta)
            -2^{-1}\varpi^{-1}\text{\rm tr}(X^2\beta)\right)
\label{eq:character-of-k-l-1-f-beta-associated-with-rho}
\end{equation}
for 
$\overline g=1_n+\varpi^{l-1}X\npmod{\frak{p}^r}\in
 K_{l-1}(\Bbb F[\overline\beta])$. 
Then the irreducible
representation $(\pi_{\beta},L^2(\Bbb W^{\prime}))$ of 
$H(\Bbb V_{\beta})$ combined with the
isomorphism $\mathcal{H}_{\beta}\,\tilde{\to}\,\Bbb H_{\beta}$ and the
group homomorphism 
$$
 \Bbb H_{\beta}\to H(\Bbb V_{\beta})
 \qquad
 ((v,z)\mapsto(v,\rho\otimes\psi_1(z)))
$$
defines an irreducible representation 
$(\pi_{\beta,\rho},L^2(\Bbb W^{\prime}))$ of $\mathcal{H}_{\beta}$ such
that $\pi_{\beta,\rho}(z)=\rho\otimes\psi_1(z)$ for all 
$z\in Z(\mathcal{H}_{\beta})$. Then $\pi_{\beta,\rho}$ combined with 
the surjection
\eqref{eq:new-surjection-of-fiber-product-to-mathcal-h-beta} defines
an irreducible representation 
$(\widetilde\pi_{\beta,\rho},L^2(\Bbb W^{\prime}))$ of 
$\Bbb G\times_{M_n(\Bbb F)}\Bbb M$ such that 
$\widetilde\pi_{\beta,\rho}(g)=\widetilde\psi_0^{-1}\cdot\widetilde\psi(g)$ for
all $g\in(\Bbb G\times_{M_n(\Bbb F)}\Bbb M)(\Bbb F[\overline\beta])$. Because 
$\widetilde\psi_0\cdot\widetilde\pi_{\beta,\rho}$ is trivial on 
$\langle\Bbb G,\Bbb M\rangle$, the representation 
$\widetilde\psi_0\cdot\widetilde\pi_{\beta,\rho}$ of 
$\Bbb G\times_{M_n(\Bbb F)}\Bbb M$ induces an irreducible
representation $(\pi_{\beta,\psi},L^2(\Bbb W^{\prime}))$ of $K_{l-1}$
such that $\pi_{\beta,\psi}(k)=\psi(k)$ for all 
$k\in K_{l-1}(\Bbb F[\overline\beta])$. 

Let us consider the action of $\mathcal{C}$ on $\Bbb H_{\beta}$ and on 
$\pi_{\beta,\rho}$. Take an $\varepsilon\in\mathcal{C}$. Let us denote
by $\overline\varepsilon\in\Bbb F[\overline\beta]^{\times}$ the image of 
$\varepsilon\in\mathcal{C}$ by the canonical surjection 
$\mathcal{C}\to\Bbb F[\overline\beta]^{\times}$. 
Because $\sigma_{\overline\varepsilon}\in Sp(\Bbb V_{\beta})$,  put
$$
 T(\varepsilon)=T_{\widehat\tau}(\sigma_{\overline\varepsilon})
 \in GL_{\Bbb C}(L^2(\Bbb W^{\prime})).
$$
Then, with the notations of the subsection 
\ref{subsec:general-setting-of-finite-symplectic-schur-multiplier}
$$
 T(\varepsilon)\circ T(\eta)
 =c_T(\overline\varepsilon,\overline\eta)\cdot
  T(\varepsilon\eta)
$$
for all $\varepsilon ,\eta\in\mathcal{C}$. 
On the other hand, for any $(x,(Y,s))\in\Bbb H_{\beta}$, we have
$$
 ([x]+Y,s)^{\varepsilon}
 =([\overline\varepsilon^{-1}x\overline\varepsilon]
               +Y+\gamma(x,\overline\varepsilon),s)
$$
where we use the notations of subsection 
\ref{subsec:general-setting-of-finite-symplectic-schur-multiplier}. 
So $\varepsilon\in\mathcal{C}$ acts on 
$(x,(Y,s))\in\Bbb H_{\beta}$ by
\begin{align*}
 (x,(Y,s))^{\varepsilon}
 &=(\overline\varepsilon^{-1}x\overline\varepsilon,
       (Y+\gamma(x,\overline\varepsilon),s))\\
 &=(x\sigma_{\overline\varepsilon},(Y,s))\cdot
   (0,(\gamma(x,\overline\varepsilon),0)).
\end{align*}
Then we have
$$
 \pi_{\beta,\rho}((X,s)^{\varepsilon})
 =\rho(\gamma([\dot X],\overline\varepsilon))\cdot
   T(\varepsilon)^{-1}\circ\pi_{\beta,\rho}(X,s)\circ
    T(\varepsilon)
$$
for all $(X,s)\in\mathcal{H}_{\beta}$. We have, with the notations of
subsection 
\ref{subsec:general-setting-of-finite-symplectic-schur-multiplier},
$$
 \rho(\gamma(x,\overline\varepsilon))
 =\widehat\tau\left(
    \langle x,v_{\overline\varepsilon}\rangle_{\overline\beta}\right)
$$
for all $x\in\Bbb V_{\beta}$. Since 
$$
 (v_{\beta},1)^{-1}(x,s)(v_{\beta},1)
 =(x,s\cdot\tau\left(\varpi^{-1}
         \langle x,v_{\beta}\rangle_{\overline\beta}\right))
$$
for all $(x,s)\in H(\Bbb V_{\beta})$, we have
$$
 \pi_{\beta,\rho}((X,s)^{\varepsilon})
 =U(\varepsilon)^{-1}\circ\pi_{\beta,\rho}(X,s)\circ
   U(\varepsilon)
$$
for all $(X,s)\in\mathcal{H}_{\beta}$ where we put
$$
 U(\varepsilon)
 =T(\varepsilon)\circ\pi_{\beta}(v_{\beta},1)
 \in GL_{\Bbb C}(L^2(\Bbb W^{\prime})).
$$
Then we have
$$
 \pi_{\beta,\psi}(\varepsilon^{-1}k\varepsilon)
       =U(\varepsilon)^{-1}\circ\pi_{\beta,\psi}(k)\circ 
        U(\varepsilon)
$$
for all $k\in K_{l-1}$ and
\begin{equation}
 U(\varepsilon)\circ U(\eta)
 =c_{\overline\beta,\rho}(\overline\varepsilon,\overline\eta)
  c_T(\overline\varepsilon,\overline\varepsilon)\cdot
  U(\varepsilon\eta)
\label{eq:explicit-formula-of-2-cocycle-of-u}
\end{equation}
for all $\varepsilon, \eta\in\mathcal{C}$. Furthermore 
$U(\varepsilon)=1$ for all 
$\varepsilon\in\mathcal{C}\cap K_{l-1}$ because 
$\mathcal{C}\cap K_{l-1}\subset K_{l-1}(\Bbb F[\overline\beta])$ and the
action of $\varepsilon\in\mathcal{C}\cap K_{l-1}$ on
$\mathcal{H}_{\beta}$ is trivial.

Now we have constructed explicitly $U(\varepsilon)$ of the formula 
\eqref{eq:conjugate-formula-of-pi-psi}, and comparing 
\eqref{eq:2-cocycle-of-u} with 
\eqref{eq:explicit-formula-of-2-cocycle-of-u}, we have
\begin{equation}
 c_U(\varepsilon,\eta)
 =c_{\overline\beta,\rho}(\overline\varepsilon,\overline\eta)\cdot
  c_T(\overline\varepsilon,\overline\eta)
\label{eq:explicit-formula-of-schur-multiplier-associated-with-rep}
\end{equation}
for all $\varepsilon ,\eta\in\mathcal{C}$. Then Theorem 
\ref{th:hypothesis-is-valid-for-separable-chara-poly-case} shows 

\begin{thm}
\label{th:original-hypothesis-is-valid-if-chara-poly-is-separable}
Hypothesis \ref{hyp:schur-multiplier-c-psi-is-trivial-for-all-psi} is
valid if the characteristic polynomial of 
$\overline\beta\in M_n(\Bbb F)$ is separable.
\end{thm}

\subsection{}
\label{subsec:regular-split-beta-in-non-dyadic-case}
In this subsection, we will assume that $\overline\beta\in M_n(\Bbb F)$
is regular split and that $\overline\beta$ is in the Jordan
canonical form \eqref{eq:beta-mod-p-is-in-jordan-canonical-form}. 
Assume also that the characteristic of $\Bbb F$ is odd. 
Then the formula 
\eqref{eq:explicit-formula-of-schur-multiplier-associated-with-rep} is
based upon the canonical polarization 
$\Bbb V_{\beta}=\Bbb W_-\oplus\Bbb W_+$ given 
at the beginning of the subsection 
\ref{subsec:splitting-characteristic-polynomial-case}. 
In this case  we have $c_T(\varepsilon,\eta)=1$
for all $\varepsilon, \eta\in\mathcal{C}$. 

Now we will reconsider the results of section 
\ref{sec:remark-on-a-result-of-g-hill}. For a diagonal matrix 
$A$ in $\Bbb F[\overline\beta]$, define an element 
$\psi_A\in X_0(\psi_{\beta})$ by
$$
 \psi_A(g)
 =\tau\left(\varpi^{-l}\text{\rm tr}(X\beta)
            -2^{-1}\varpi^{-1}\text{\rm tr}(X^2\beta)
            +\varpi^{-1}\text{\rm tr}(XA)\right)
$$
for $g=\overline{1_n+\varpi^{l-1}X}\in K_{l-1}(W)$. Then 
$\psi=\psi_A|_{K_{l-1}(\Bbb F[\overline\beta])}\in Y(\psi_{\beta})$
corresponds to the additive character $\rho=\rho_A$ of 
$\Bbb F[\overline\beta]$ by 
\eqref{eq:character-of-k-l-1-f-beta-associated-with-rho}. For such a 
$\psi\in Y(\psi_{\beta})$, we have $c(\psi)=1$ in
$H^2(\mathcal{C},\Bbb C^{\times})$ by Proposition 
\ref{prop:invariant-extension-of-additive-character} and by 
\eqref{eq:explicit-formula-of-schur-multiplier-associated-with-rep}. 
On the other hand the number
of the diagonal elements in $\Bbb F[\overline\beta]$ is $q^r$, and the
remark just after Proposition 
\ref{prop:k-l-1-conjugate-of-psi-in-x-0-psi-beta} shows that the
characters 
$\psi_A$ with diagonal matrix $A\in\Bbb F[\overline\beta]$ are the
complete set 
of the representatives of the $K_{l-1}$-orbits in $X(\psi_{\beta})$
which contains some element of $X_0(\psi_{\beta})$. This means that 
Hypothesis \ref{hyp:schur-multiplier-c-psi-is-trivial-for-all-psi} is
valid if $\overline\beta\in M_n(\Bbb F)$ is regular split and
semi-simple (that is the case where the argument in the proof of Theorem
4.6 of \cite{Hill1995-2} works), 
but this is a special case of Theorem 
\ref{th:original-hypothesis-is-valid-if-chara-poly-is-separable}. 

The examples 
\ref{ex:trivial-2-cocycle-in-n-is-2}, 
\ref{ex:trivial-2-cocycle-in-n-is-3}, 
\ref{ex:trivial-2-cocycle-in-n-is-4} show

\begin{thm}\label{th:original-hypothesis-is-valid-in-split-small-n}
Hypothesis \ref{hyp:schur-multiplier-c-psi-is-trivial-for-all-psi} is
valid if $n_i\leq 4$ ($i=1,\cdots,r$) and $\text{\rm ch}\,\Bbb F>7$.
\end{thm}

These results strongly suggest that Hypothesis 
\ref{hyp:schur-multiplier-c-psi-is-trivial-for-all-psi} is valid if 
$\overline\beta\in M_n(\Bbb F)$ is regular and the characteristic
of $\Bbb F$ is big enough.
We will discuss on this subject further in the forthcoming paper.